\documentclass[preprint,12pt]{article}
\usepackage{amsmath,amssymb,amsthm}
\usepackage{graphicx}
\usepackage{subfigure}
\usepackage{color}
\usepackage{amssymb}
\usepackage{setspace}
\usepackage{graphicx}
\usepackage{float}
\usepackage[latin1]{inputenc}
\newtheorem{thm}{Theorem}[section]

 \newtheorem{lem}{Lemma}[section]
 \newtheorem{prop}{Proposition}[section]
 \newtheorem{defn}{Definition}[section]
\newtheorem{rem}{Remark}[section]

\def\d{\partial}
\def\ddj{\dot \Delta_j}

\def\tilde{\widetilde}
\def\hat{\widehat}

\def \zr1{$z_{R,1}$}
\def \zr2{z_{R,2}}
\def \zi1{z_{I,1}}
\def \zi2{z_{I,2}}

\def\cP{{\mathcal P}}
\def\cQ{{\mathcal Q}}

\def\Scho{{Schr$\ddot{o}$dinger}}

\renewcommand{\div}{\mbox{\rm div}\;\!}

\newcommand\R{\mathbb{R}}

\begin{document}
\title{Global dynamics of large solution for the compressible Navier-Stokes-Korteweg equations}
\author{Zihao Song}
\date{\it \small Research Institute for Mathematical Sciences, Kyoto University\\
Kyoto, 606-8007, Japan\\
 E-mail: szh1995@nuaa.edu.cn}

\maketitle
\begin{abstract}
In this paper, we study the Navier-Stokes-Korteweg equations governed by the evolution of compressible fluids with capillarity effects. We first investigate the global well-posedness of solution in the critical Besov space for large initial data. Contrary to pure parabolic methods in Charve, Danchin and Xu \cite{CDX}, we also take the strong dispersion due to large capillarity coefficient $\kappa$ into considerations. By establishing a dissipative-dispersive estimate, we are able to obtain uniform estimates and incompressible limits in terms of $\kappa$ simultaneously.

Secondly, we establish the large time behaviors of the solution. We would make full use of both parabolic mechanics and dispersive structure which implicates our decay results without limitations for upper bound of derivatives while requiring no smallness for initial assumption.
\end{abstract}

\section{Introduction}\setcounter{equation}{0}
In this paper, we investigate the compressible Navier-Stokes-Korteweg model, whose theory formulation was first introduced by Van der Waals \cite{V}, Korteweg \cite{K}. The model aims to study the dynamics of a liquid-vapor mixture in the Diffuse Interface (DI) approach, where the phase changes are described through the variations of the density. The rigorous derivation of the corresponding equations for the compressible Navier-Stokes-Korteweg system (NSK) reads as:
\begin{equation}
\left\{
\begin{array}{l}\partial_{t}\rho+\mathrm{div}(\rho u)=0,\\ [1mm]
 \partial_{t}(\rho u)+\mathrm{div}(\rho u \otimes u)+\nabla P(\rho)=\mathcal{A}(u)+\mathrm{div}K(\rho).\\[1mm]
 \end{array} \right.\label{1.1}
\end{equation}
Here, $\rho=\rho(t,x)\in \mathbb{R}_{+}$ and $u=u(t,x)\in \mathbb{R}^{d}(d\geq2)$ are the unknown functions on $[0,+\infty)\times \mathbb{R}^{d}$, which stand for the density
and velocity field of a fluid, respectively. We neglect the thermal fluctuation so that the pressure $P=P(\rho)$ reduces to a function of $\rho$ only. The notation
$\mathcal{A}u$ is given by $\mathcal{A}u=\mathrm{div}(2\mu D(u))+\nabla(\lambda\mathrm{div}u)$ with $D(u)\triangleq\frac{1}{2}(\nabla u+\nabla^T u)$, where the Lam\'{e} coefficients $\lambda$ and $\mu$ (the \textit{bulk} and \textit{shear viscosities}) are density-dependent functions, respectively. In order to ensure the uniform ellipticity of $\mathcal{A}u$,
they are assumed to satisfy
$$\lambda>0, \nu\triangleq\lambda+2\mu>0.$$
In the following, the Korteweg tensor is given by (see \cite{BDDJ})
$$\mathrm{div}K(\rho)=\kappa\rho\nabla\big(m(\rho)\Delta\rho+\frac{1}{2}m'(\rho))|\nabla\rho|^{2}\big).$$
Here, $\kappa\in \mathbb{R}^{+}$ represents the capillarity coefficient while $m(\rho)$ may depend on $\rho$ in general. The initial condition of System (\ref{1.1}) is prescribed by
\begin{equation}\label{1.2}
\left(\rho,u\right)|_{t=0}=\left(\rho _{0}(x),u_{0}(x)\right),\  x\in \R^{d}.
\end{equation}
In this paper, we investigate the Cauchy problem  \eqref{1.1}-\eqref{1.2}, where initial data tends to a constant equilibrium $(\rho^{\ast},0)$ with $\rho^{\ast}>0$.

There were fruitful mathematical results on the compressible fluid models of Korteweg type in the past thirty years. Hattori and Li  \cite{HL,HL1} obtained global smooth solutions for initial data close enough to a stable equilibrium $(\rho^{\ast},0)$. Bresch, Desjardins and Lin \cite{DDL} established the global existence of weak solutions in a periodic or strip domain. However, the uniqueness problem of weak solutions has not been solved. A natural way of dealing with the uniqueness is to find a functional setting as large as possible in which the existence and uniqueness hold. This idea is closely linked with the concept of scaling invariance space, which has been successfully employed by Fujita-Kato \cite{FK}, Cannone \cite{Can} and Chemin \cite{C} for incompressible Navier-Stokes equations. Danchin \cite{D} first developed the idea of scaling invariance in compressible Navier-Stokes equations. Note the fact that (NSK) is invariant by the transformation
\begin{eqnarray*}
\rho(t,x)\leadsto \rho(\ell ^2t,\ell x),\quad
u(t,x)\leadsto \ell u(\ell^2t,\ell x), \ \ \ell>0
\end{eqnarray*}
up to a change of the pressure term  $P$ into $\ell^2P$, Danchin and Desjardins \cite{DD} investigated the global well-posedness of perturbation solutions for (NSK) in critical Besov spaces. Charve, Danchin and Xu \cite{CDX} investigated the global existence and Gevrey analyticity of \eqref{1.1} in more general critical $L^p$ framework. Chikami and Kobayashi \cite{CK} studied the optimal time-decay estimates in the $L^2$ critical Besov spaces. Kawashima, Shibata and Xu \cite{KSX} investigated the dissipation effect of Korteweg tensor with the density-dependent capillarity and developed the $L^p$ energy methods (independent of spectral analysis), which leads to the optimal time-decay estimates of strong solutions. Murata and Shibata \cite{MS} addressed a totally different statement on the global existence of strong solutions to \eqref{1.1} in Besov spaces, where the maximal $L^p$-$L^q$ regularity was mainly employed.

One important direction of studying fluid models is to investigate the singular limit in terms of different physical parameters. For Navier-Stokes-Korteweg equations, the zero mach number limit is studied in \cite{LY,JX}. As for vanishing limits for capillarity coefficient, one could refer to results given in \cite{BYZ,BH,HYZ}. Recently, some multi-scale problems for Korteweg system with high rotation were considered in Fanelli \cite{F}.

However, there are few results concern the case with large Korteweg coefficient $\kappa$, which may reflect strong capillarity effect for the fluid, and the corresponding converge process. In this paper, we would investigate this situation where the perturbation system admits not only parabolic mechanics, but also dispersive structures because of its strong capillarity effect. Our first goal is to prove the global well-posedness with initial velocity arbitrary large in the critical Besov space under $\kappa$ large enough and the global existence of the incompressible Navier-Stokes equation (INS) which reads
\begin{equation*}\label{incompressible}
\qquad\qquad\left\{
\begin{array}{l}\partial_{t}v-\mu\Delta v+v\cdot\nabla v+\nabla\pi=0,\\ [1mm]
\mathrm{div}v=0,\qquad\qquad\qquad\qquad\qquad\qquad\qquad\qquad\qquad\mathrm{(INS)}\\ [1mm]
v|_{t=0}=v_{0}(x).
 \end{array} \right.
\end{equation*}
Also, we will show the solution we construct converges to the solution of the incompressible Navier-Stokes equation. Moreover, we would establish the optimal decay rates for the solution of any order derivatives, but without asking any smallness for the initial assumption.

\section{Reformulation and main results}
Without loss of generality, we shall fix the equilibrium of density to be $\rho^{\ast}=1$ and assume $P'(1)=\mu(1)=\lambda(1)=m(1)=1$. Denote the density fluctuation by $a=\sqrt{\kappa}\mathcal{L}(\rho)$ where
$$\mathcal{L}(\rho)=\int\limits^{\rho}_{1}\sqrt{\frac{m(s)}{s}}ds.$$
Then the density $\rho$ is also given by $\rho=\mathcal{L}^{-1}(\kappa^{-\frac{1}{2}}a)$. A simple calculation leads us to the following perturbation problem:
\begin{equation}
\left\{
\begin{array}{l}\partial_{t}a+\sqrt{\kappa}\mathrm{div}u=f,\\ [1mm]
 \partial_{t}u-\bar{\mathcal{A}}u+\frac{1}{\sqrt{\kappa}}\nabla a-\sqrt{\kappa}\nabla\Delta a= g,\\[1mm]
(a,u)|_{t=0}=(a_{0},u_{0}),\\[1mm]
 \end{array} \right.\label{linearized}
\end{equation}
with
$$a_{0}\triangleq \sqrt{\kappa}\mathcal{L}(\rho_{0}) \quad \mbox{and} \quad \bar{\mathcal{A}}u=\Delta u+2\nabla\mathrm{div}u.$$
One can write the nonlinear terms $f=-u\cdot\nabla a-\sqrt{\kappa}\tilde{\psi}(\kappa^{-\frac{1}{2}}a)\mathrm{div}u$ and $g=\sum\limits^{5}_{i=1} g_{i}$ as follows:
\begin{equation}
\left\{
\begin{array}{l}g_{1}=u\cdot\nabla u,\\ [1mm]
g_{2}=\big(1+\tilde{Q}(\kappa^{-\frac{1}{2}}a)\big)\big(\mathrm{div}(2\tilde{\mu}(\kappa^{-\frac{1}{2}}a) D(u))+\nabla(\tilde{\lambda}(\kappa^{-\frac{1}{2}}a)\mathrm{div}u)\big),\\ [1mm]
g_{3}=\kappa^{-\frac{1}{2}}\tilde{G}(\kappa^{-\frac{1}{2}}a)\nabla a,\\ [1mm]
g_{4}=\sqrt{\kappa}\nabla(\tilde{\psi}(\kappa^{-\frac{1}{2}}a)\Delta a),\\ [1mm]
g_{5}=\frac{1}{2}\nabla(|\nabla a|^{2}),\\ [1mm]
 \end{array} \right.\label{nonlinear}
\end{equation}
where $\tilde{Q}(a)=Q\circ\mathcal{L}^{-1}(a)$ (similarly definitions for $\tilde{\mu}, \tilde{\lambda}, \tilde{G}, \tilde{\psi}$) with
\begin{eqnarray*}
Q(\rho)&=&\frac{1}{\rho}-1,\quad G(\rho)=\frac{P'(\rho)}{\sqrt{m(\rho)\rho}}-1,\quad \bar{\mu}(\rho)=\mu(\rho)-1,\\
&&\bar{\lambda}(\rho)=\lambda(\rho)-1,\quad \psi(\rho)=\sqrt{m(\rho)\rho}-1.
\end{eqnarray*}
In our analysis, those functions are assumed to be smooth and vanishing at zero whose exact values will not matter. Now, denote Leray projector by $\mathcal{P}=\mathrm{Id}-\nabla(-\Delta)^{-1} \mathrm{div}$, our first main result is stated as follows.
\begin{thm}\label{thm2.1}
Assume $d\geq2$ and $\kappa>1$. Let $(\frac{1}{\sqrt{\kappa}}a_{0},\nabla a_{0}, u_{0})\in\dot{B}^{\frac{d}{2}-1}_{2,1}$. If $\mathrm{(INS)}$ with initial data $\cP u_{0}$ admits a global solution $v$ satisfying for any $T\in(0,\infty)$ that
\begin{eqnarray}
v\in\mathcal{C}_{T}(\dot{B}^{\frac{d}{2}-1}_{2,1})\cap L^1_{T}(\dot{B}^{\frac{d}{2}+1}_{2,1}),
\end{eqnarray}
then there exists a positive $\kappa_{0}$ depending on the initial data such that for all $\kappa\geq\kappa_{0}$, the Cauchy problem (\ref{linearized})-(\ref{nonlinear}) admits a unique global-in-time solution $(a,u)$ satisfying
\begin{equation}\label{bound}
(\frac{1}{\sqrt{\kappa}}a,\nabla a, u)\in\mathcal{C}_{T}(\dot{B}^{\frac{d}{2}-1}_{2,1})\cap L^1_{T}(\dot{B}^{\frac{d}{2}+1}_{2,1}).
\end{equation}
Furthermore, $(\nabla a,\cQ u)$ tends to 0 in $\tilde{L}^{2}_{T}(\dot{B}^{\frac{d}{p}}_{p,1})$, $\cP u$ tends to $v$ in $L^\infty_{T}(\dot{B}^{\frac{d}{2}-1}_{2,1})\cap L^1_{T}(\dot{B}^{\frac{d}{2}+1}_{2,1})$ satisfying
\begin{equation}\label{converge}
\|(\nabla a,\mathcal{Q}u)\|_{\tilde{L}^{2}_{T}(\dot{B}^{\frac{d}{p}}_{p,1})}\leq C\kappa^{-\delta};\quad
\|\cP u-v\|_{L^\infty_{T}(\dot{B}^{\frac{d}{2}-1}_{2,1})\cap L^1_{T}(\dot{B}^{\frac{d}{2}+1}_{2,1})}\leq C\kappa^{-\frac{\delta}{2}},
\end{equation}
where $(\delta,p)$ satisfies
\begin{equation}\label{pair}
(\delta,p)=\left\{
\begin{array}{l}(\frac{1}{4},\frac{2d}{d-2}),\qquad \quad d\geq3\\ [1mm]
(\frac{1}{4}-\varepsilon,\frac{1}{2\varepsilon}),\qquad  d=2
 \end{array} \right.
\end{equation}
with positive $\varepsilon>0$ sufficient small.
\end{thm}

\begin{rem}
The motivation of Theorem \ref{thm2.1} initiates from observation that linearized system admits a \Scho\, type dispersive structure brought by the three order term. Different from pure parabolic methods in \cite{CDX}, we would take both dissipation and dispersion into consideration and establish a Strichartz type estimate which implicates some smallness in terms of the coefficient $\kappa$.

Instead of classical perturbation variable $\rho-\rho^*$, we apply the modified density $a(\rho)$ which defines a diffeomorphism from $\mathbb{R}^+$ onto a small neighborhood
of 0. This fact enables us to symmtric those quasi-linear nonlinearities don't contain smallness of $\kappa$, and offers us some nice cancellations in energy estimates.
\end{rem}

\begin{rem}
The global-in-time well-posedness of $\mathrm{(INS)}$ with large initial data has been fixed in many directions. In 2D case, $\mathrm{(INS)}$ is always globally well-posed, see Leray \cite{L}, Serrin \cite{S}. For the case $d=3$, the global well-posedness is ensured with initial data with special geometry structure, like axisymmetric \cite{LMNP}, or those vary slowly in one space direction, see \cite{CG}.
\end{rem}

Secondly, we are going to establish the decay rates for the solution we establish above. Denote the multiplier $\Lambda^{\alpha}f\triangleq \mathcal{F}^{-1}(|\xi|^{\alpha}\hat{f}(\xi))$, with $\alpha\in\mathbb{R}$, our decay results are given as follows:
\begin{thm}\label{thm2.2}
Let $d\geq2$, $1-\frac{d}{2}<\sigma\leq\frac{d}{2}$. Let $(a,u)$ be the global-in-time solution constructed in Theorem \ref{thm2.1}. Assume initial data additionally satisfies
\begin{eqnarray}
(\frac{1}{\sqrt{\kappa}}a_{0},\nabla a_{0}, u_{0})\in\dot{B}^{-\sigma}_{2,\infty},
\end{eqnarray}
If the solution $v$ of $\mathrm{(INS)}$ with initial data $\cP u_{0}$ shares the decay estimates for any index $|\alpha|>-\sigma$ such that
\begin{eqnarray}\label{incompressible decay}
\|\Lambda^{\alpha}v\|_{L^2}\lesssim t^{-\frac{|\alpha|}{2}-\frac{\sigma}{2}},
\end{eqnarray}
then $(a,u)$ holds for any $t\geq1$
\begin{eqnarray}
\|\Lambda^{\alpha}(\frac{1}{\sqrt{\kappa}}a,\nabla a, u)\|_{L^2}\lesssim t^{-\frac{|\alpha|}{2}-\frac{\sigma}{2}}.
\end{eqnarray}
\end{thm}

\begin{rem}
The decay of solutions for heat equations or $\mathrm{(INS)}$ originated from series works of Schonbek \cite{S-CPDE,S-ARMA} and later developed by \cite{Lo,NS,W}. On the other hand, the decay of high order derivatives is also considered by \cite{S4,OT}. We remark that according to \cite{Lo}, one is able to prove (\ref{incompressible decay}) once $\cP u_{0}\in\dot{B}^{-\sigma}_{2,\infty}$.

Taking $\sigma=\frac{d}{2}$ ($L^1\hookrightarrow\dot{B}^{-\sigma}_{2,\infty} $), Theorem \ref{thm2.2} and Sobolev embedding reflect classical decay rates of any order derivatives as heat equations where
$$\|\Lambda^{\alpha}(\frac{1}{\sqrt{\kappa}}a,\nabla a, u)\|_{L^r}\lesssim t^{-\frac{d}{2}(1-\frac{1}{r})-\frac{\alpha}{2}},\quad r\geq2, t\geq1$$
with initial assumption arbitrary large.
\end{rem}

\begin{rem}
Results of Theorem \ref{thm2.1} and \ref{thm2.2} could also be obtained for those irrotational fluids. At this moment, the velocity could be written as a potential while the incompressible part disappears, and one may remove those assumptions concerns $\cP u_{0}$ in the main Theorems.
\end{rem}

\subsection{Strategy}
In order to understand the proof of Theorems \ref{thm2.1} and Theorems \ref{thm2.2} well, we make a formal spectral analysis to the linearized system of \eqref{linearized} which can be written in terms of
the divergence-free part $\mathcal{P}u$ and the compressible one $\mathcal{Q}u=(I-\mathcal{P})u$:
\begin{equation}
\left\{
\begin{array}{l}\partial_{t}a+\sqrt{\kappa}\mathrm{div} \mathcal{Q}u=f,\\ [1mm]
 \partial_{t}\mathcal{Q}u-2\Delta\mathcal{Q}u+\frac{1}{\sqrt{\kappa}}\nabla a-\sqrt{\kappa}\nabla\Delta a=\mathcal{Q}g,\\ [1mm]
 \partial_{t}\mathcal{P}u-\Delta\mathcal{P}u=\mathcal{P}g.
 \end{array} \right.\label{spectrum}
\end{equation}

It is clearly that the incompressible part $\mathcal{P}u$ just satisfies an ordinary heat equation. Regarding the compressible part $\mathcal{Q}u$, it is convenient to
introduce $\mathcal{V}\triangleq \Lambda^{-1}\div u$ where the new variable $(a, \mathcal{V})$ satisfies the coupling $2\times 2$ system:
\begin{equation}\label{R-E69}
\left\{\begin{array}{l}\d_ta+\sqrt{\kappa}\Lambda \mathcal{V}=0,\\[1ex]
\d_t\mathcal{V}-2\Delta \mathcal{V}-\frac{1}{\sqrt{\kappa}}\Lambda a-\sqrt{\kappa}\Lambda^{3} a=0.
\end{array}\right.
\end{equation}
Taking the Fourier transform with respect to $x\in \mathbb{R}^{d}$ leads to
\begin{equation}
\frac{d}{dt}\left(\begin{array}{c}
\hat{a} \\
\hat{\mathcal{V}} \\
\end{array}\right)
=A(\xi)\left(\begin{array}{c}
\hat{a} \\
\hat{\mathcal{V}} \\
\end{array}\right)
\quad \mbox{with}\quad A(\xi)=\left(
\begin{array}{cc}0 & -\sqrt{\kappa}|\xi| \\
\frac{1}{\sqrt{\kappa}}|\xi|+\sqrt{\kappa}|\xi|^3 & -2|\xi|^2 \\
\end{array}\right),
\end{equation}
where $\xi\in \mathbb{R}^{d}$ is the Fourier variable. It is not difficult to check that
$$\lambda_{\pm}=|\xi|^2\pm\sqrt{(1-\kappa)|\xi|^4-|\xi|^2}.$$
Consequently, it is found that the coupling system presents both dissipation and dispersion. In particular, while $\kappa>1$, the corresponding dispersive structure is closely linked with the so called the Gross-Patavskii equation which reads as
\begin{equation}\label{GP}
i\partial_{t}\psi+\Delta\psi-2\mathrm{Re}\psi=F(\psi).
\end{equation}
Based on dispersive estimates of Gross-Patavskii equations established in Gustafson, Nakanishi and Tsai \cite{GNT1,GNT2}, we would take not only optimal smooth effect of heat kernel (see $\phi_{1}$ in Figure 1), but also Strichartz estimates ($\phi_{2}$ in Figure 1) into considerations, which enables us to establish a dissipative-dispersive estimate in $\tilde{L}^r_{T}(\dot{B}^{s}_{p,q})$ space (see Lemma \ref{Strichartz estimates}), i.e. area $\Psi$ in Figure 1 and obtain some smallness in terms of sufficient large dispersion coefficient $\kappa$.

\begin{figure}[H]
\centering
\includegraphics[width=13.8cm]{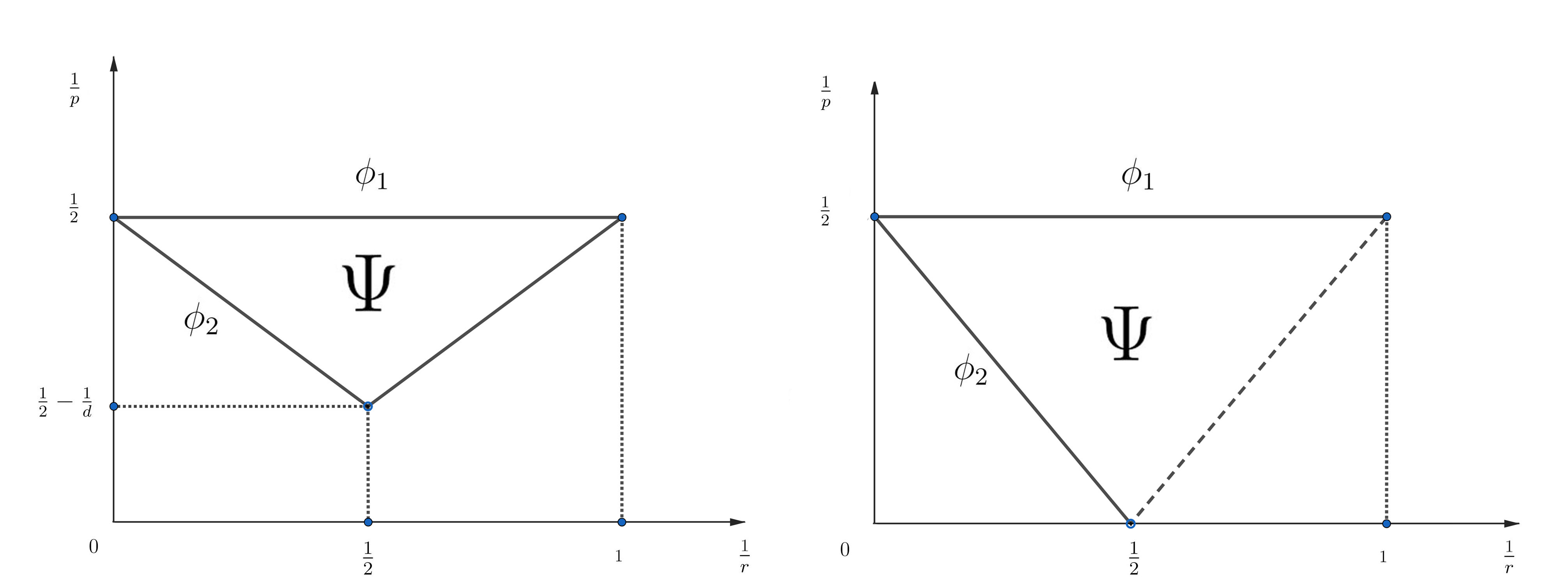}\\
\caption{$ d\geq3\qquad\qquad\qquad\qquad\qquad\qquad\qquad d=2 \qquad\qquad\qquad$}
\end{figure}

Based on the dissipative-dispersive estimates and uniform control of $v$, we are able to give the uniform estimates with initial data arbitrary large. Moreover, the incompressible limits are obtained according to different dimensions.

Our second result would focus on the decay rates. Observe the parabolic behavior of the perturbation system, we expect to establish decay for even high order derivatives. Different from the energy method \cite{KSX} or Gevrey method established in \cite{SX}, our decay derivation consists of two steps:
\begin{itemize}
\item Evolution of negative Besov norm $\|(a,u)\|_{\dot{B}^{-\sigma}_{2,\infty}}$;
\item Decay estimates for high order derivatives.
\end{itemize}

The main difficulties come from bounding nonlinearities since initial data doesn't indicates any smallness. The Strichartz estimates would again be our main tool to offer smallness in terms of $\kappa$ where we shall carefully deal with different frequency interactions.

Throughout the paper, $C>0$ stands for a generic ``constant". For brevity, $f\lesssim g$ means that $f\leq Cg$. It will also be understood that $\|(f,g)\|_{X}=\|f\|_{X}+\|g\|_{X}$ for all $f,g\in X$.

Finally, the rest of this paper unfolds as follows: in Section 3, we prove the global well-posedness and incompressible limits of solutions. Section 4 is devoted to giving the optimal decay estimates. In the last section (``Appendix"), we recall the classical Littlewood-Paley theory.

\section{The global well-posedness and incompressible limit}\setcounter{equation}{0}
We focus on the proof of Theorem \ref{thm2.1} in this section. Before the proof, we shall present some notations concern functional spaces:
\begin{eqnarray}\label{E}
\mathcal{E}_{T}\triangleq\|(\frac{1}{\sqrt{\kappa}}a,\nabla a, \cQ u)\|_{L^{\infty}_{T}(\dot{B}^{\frac{d}{2}-1}_{2,1})}
+\|(\frac{1}{\sqrt{\kappa}}a,\nabla a, \cQ u)\|_{L^1_{T}(\dot{B}^{\frac{d}{2}+1}_{2,1})};
\end{eqnarray}
\begin{eqnarray}\label{W}
\mathcal{W}_{T}\triangleq\|\cP u-v\|_{L^{\infty}_{T}(\dot{B}^{\frac{d}{2}-1}_{2,1})}+\|\cP u-v\|_{L^{1}_{T}(\dot{B}^{\frac{d}{2}+1}_{2,1})}.
\end{eqnarray}
Also, the following Strichartz space is defined:
\begin{eqnarray}\label{D}
\mathcal{D}_{T}\triangleq\kappa^{\delta}\|(\sqrt{\frac{\kappa^{-1}-\Delta}{-\Delta}}\nabla a,\mathcal{Q}u)\|_{\tilde{L}^{2}_{T}(\dot{B}^{\frac{d}{p}}_{p,1})},
\end{eqnarray}
here $(\delta, p)$ coincides with (\ref{pair}). Moreover, we shall define $\mathcal{V}_{t}$ to present functional space for solutions of (INS):
\begin{eqnarray}\label{V}
\mathcal{V}_{T}\triangleq\|v\|_{L^{\infty}_{T}(\dot{B}^{\frac{d}{2}-1}_{2,1})}+\|v\|_{L^{1}_{T}(\dot{B}^{\frac{d}{2}+1}_{2,1})}.
\end{eqnarray}

We shall first establish the priori estimates of $\mathcal{E}_{T},\mathcal{W}_{T},\mathcal{D}_{T}$ for the solution of (\ref{linearized})-(\ref{nonlinear}) under the critical regularity. Moreover, we show such a priori estimate is uniform in terms of arbitrary large initial data and $\mathcal{V}_{T}$ if the capillary effect is strong enough, i.e. $\kappa\geq\kappa_{0}$ for some $\kappa_{0}$ sufficient large. Then the well-posedness is ensured by those local results established in \cite{DD} and a bootstramp argument.

\subsection{{\bf Priori estimates}}
This subsection is devoted to proving the priori estimate. Without loss of generality, we always assume that $\kappa>1$. Our main result is to prove
\begin{prop}\label{uniform estimates}
Let $d\geq2$ and $\kappa>1$. Suppose that $(a,u)$ is the solution of system (\ref{1.1})-(\ref{1.2}). Then the following priori estimates holds true
\begin{multline}\label{1}
\mathcal{D}_{T}+\mathcal{E}_{T}\leq C\big( \|(\frac{1}{\sqrt{\kappa}}a_{0},\nabla a_{0}, \cQ u_{0})\|_{\dot{B}^{\frac{d}{2}-1}_{2,1}}+\mathcal{V}^{2}_{T}+\mathcal{W}^{2}_{T}+\kappa^{-\delta}(1+\mathcal{E}_{T})(\mathcal{D}_{T}+\mathcal{E}_{T})^2\big)
\end{multline}
while
\begin{eqnarray}\label{2}
\mathcal{W}_{T}\leq Ce^{\tilde{C}\mathcal{V}_{T}}\big(\mathcal{W}^{2}_{T}+\kappa^{-\frac{\delta}{2}}(1+\mathcal{E}_{T})(\mathcal{D}_{T}+\mathcal{E}_{T}+\mathcal{V}_{T}+\mathcal{W}_{T})^{2}\big).
\end{eqnarray}
\end{prop}

\begin{proof}
We begin with establishing the linear estimates by bounding the dissipative-dispersive coupling system and incompressible part respectively.

{\it Step1: \underline{Dissipative estimates in energy framework for $(a,\mathcal{Q}u)$}.}

The coupling system of $(a,\mathcal{Q}u)$ could be written as
\begin{equation}
\left\{
\begin{array}{l}\partial_{t}a+\sqrt{\kappa}\mathrm{div} \mathcal{Q}u=f,\\ [1mm]
 \partial_{t}\mathcal{Q}u-2\Delta u+\frac{1}{\sqrt{\kappa}}\nabla a-\sqrt{\kappa}\nabla\Delta a=\mathcal{Q}g.
 \end{array} \right.\label{spectrum}
\end{equation}
Imposing the gradient on the continuity equation, then Fourier localization leads us to
\begin{equation}
\left\{
\begin{array}{l}\partial_{t}\ddj \nabla a+\sqrt{\kappa}\ddj\Delta\mathcal{Q}u=\ddj\nabla f,\\ [1mm]
 \partial_{t}\ddj \mathcal{Q}u-2\ddj\Delta u+\frac{1}{\sqrt{\kappa}}\ddj\nabla a-\sqrt{\kappa}\ddj\nabla\Delta a=\ddj\mathcal{Q} g.
 \end{array} \right.\label{spectrum}
\end{equation}
Denote $(a_{j},u_{j})=(\ddj a,\ddj u)$ and $(f_{j},g_{j})=(\ddj f,\ddj g)$, $L^2$ inner product yields
\begin{eqnarray}\label{energy1}
\frac{1}{2}\frac{d}{dt}\|\nabla a_{j}\|^{2}_{L^{2}}+\sqrt{\kappa}\int_{\mathbb{R}^{d}}\Delta\mathcal{Q}u_{j}\nabla a_{j}dx=\int_{\mathbb{R}^{d}}\nabla f_{j}\nabla a_{j}dx.
\end{eqnarray}
\begin{multline}
\frac{1}{2}\frac{d}{dt}\|\mathcal{Q}u_{j}\|^{2}_{L^{2}}+2\|\nabla\mathcal{Q}u_{j}\|^{2}_{L^{2}}+\frac{1}{\sqrt{\kappa}}\int_{\mathbb{R}^{d}}\nabla a_{j}\mathcal{Q}u_{j}dx\\-\sqrt{\kappa}\int_{\mathbb{R}^{d}}\nabla\Delta a_{j}\mathcal{Q}u_{j}dx
=\int_{\mathbb{R}^{d}} g_{j} \mathcal{Q}u_{j}dx.
\end{multline}
\begin{multline}
\kappa^{-\frac{1}{2}}\frac{1}{2}\frac{d}{dt}\int_{\mathbb{R}^{d}}\mathcal{Q}u_{j}\nabla a_{j}dx-\|\nabla\mathcal{Q}u_{j}\|^{2}_{L^{2}}-\frac{2}{\sqrt{\kappa}}\int_{\mathbb{R}^{d}}\mathcal{Q}u_{j}\nabla a_{j}dx+\kappa^{-1}\| \nabla a_{j}\|^{2}_{L^{2}}+\|\Delta a_{j}\|^{2}_{L^{2}}\\
=\kappa^{-\frac{1}{2}}\int_{\mathbb{R}^{d}}\nabla f_{j}\mathcal{Q}u_{j}dx+\kappa^{-\frac{1}{2}}\int_{\mathbb{R}^{d}} g_{j}\nabla a_{j}dx.
\end{multline}
Now to deal with quasi-linear nonlinearity $-\sqrt{\kappa}\tilde{\psi}(\kappa^{-\frac{1}{2}}a)\mathrm{div}u$ in $f$ and $\sqrt{\kappa}\nabla(\tilde{\psi}(\kappa^{-\frac{1}{2}}a)\Delta a)$ in $g$, we have
\begin{multline*}
-\int_{\mathbb{R}^{d}}\ddj\nabla(\tilde{\psi}(\kappa^{-\frac{1}{2}}a)\mathrm{div}u)\nabla a_{j}dx=\int_{\mathbb{R}^{d}} [\ddj,\tilde{\psi}(\kappa^{-\frac{1}{2}}a)]\mathrm{div}u\Delta a_{j}dx\\+\int_{\mathbb{R}^{d}} \tilde{\psi}(\kappa^{-\frac{1}{2}}a)\mathrm{div}\mathcal{Q}u_{j}\Delta a_{j}dx
\end{multline*}
while it holds
\begin{multline*}
\int_{\mathbb{R}^{d}}\ddj\nabla(\tilde{\psi}(\kappa^{-\frac{1}{2}}a)\Delta a)\mathcal{Q}u_{j}dx=-\int_{\mathbb{R}^{d}} [\ddj,\tilde{\psi}(\kappa^{-\frac{1}{2}}a)]\Delta a\mathrm{div}\cQ u_{j}dx\\-\int_{\mathbb{R}^{d}} \tilde{\psi}(\kappa^{-\frac{1}{2}}a)\mathrm{div}\Delta a_{j}\mathcal{Q}u_{j}dx.
\end{multline*}
Notice that it also holds
\begin{eqnarray}\label{energy2}
\frac{1}{2\kappa}\frac{d}{dt}\|a_{j}\|^{2}_{L^{2}}+\frac{1}{\sqrt{\kappa}}\int_{\mathbb{R}^{d}}\mathrm{div}\mathcal{Q}u_{j}a_{j}dx=-\frac{1}{\kappa}\int_{\mathbb{R}^{d}} f_{j}a_{j}dx.
\end{eqnarray}
We immediately deduced from (\ref{energy1})-(\ref{energy2}) that
\begin{multline}\label{localization estimates}
\frac{d}{dt}\|(\frac{1}{\sqrt{\kappa}}a_{j},2^{j}a_{j},u_{j})\|_{L^{2}}+c2^{2j}\|(\frac{1}{\sqrt{\kappa}}a_{j},2^{j}a_{j},u_{j})\|_{L^{2}}\leq \|(\frac{1}{\sqrt{\kappa}}f_{j},g_{j}-g_{4,j})\|_{L^{2}}\\
+2^{j}\|\ddj(u\cdot\nabla a)\|_{L^{2}}+2^{j}\|[\ddj,\tilde{\psi}(\kappa^{-\frac{1}{2}}a)](\mathrm{div}u+\Delta a)\|_{L^{2}}.
\end{multline}

Hence taking $\dot{B}^{\frac{d}{2}-1}_{2,1}$ norm and integral on $t$ lead us to
\begin{multline}\label{dissipation}
\|(\frac{1}{\sqrt{\kappa}}a,\nabla a, \cQ u)\|_{L^{\infty}_{T}(\dot{B}^{\frac{d}{2}-1}_{2,1})\cap L^1_{T}(\dot{B}^{\frac{d}{2}+1}_{2,1})}
\lesssim \|(\frac{1}{\sqrt{\kappa}}a_{0},\nabla a_{0}, \cQ u_{0})\|_{\dot{B}^{\frac{d}{2}-1}_{2,1}}+\|u\cdot\nabla a\|_{L^{1}_{T}(\dot{B}^{\frac{d}{2}}_{2,1})}\\
+\|(\frac{1}{\sqrt{\kappa}}f,\cQ g-\cQ g_{4})\|_{L^{1}_{T}(\dot{B}^{\frac{d}{2}-1}_{2,1})}
+\sqrt{\kappa}\int^{T}_{0}\|[\ddj,\tilde{\psi}(\kappa^{-\frac{1}{2}}a)](\mathrm{div}u+\Delta a)\|_{\dot{B}^{\frac{d}{2}}_{2,1}}ds.
\end{multline}

{\it Step2: \underline{Dissipative-dispersive estimates for $(a,\mathcal{Q}u)$}.}

Secondly we shall establish dissipative-dispersive estimates of $(a,\mathcal{Q}u)$ where one could find some smallness in $\tilde{L}^{r}_{T}(\dot{B}^{s}_{p,1})$ ($1<r<\infty, 2<p<\infty$) space when dispersive coefficient $\kappa$ is large. Inspired by \cite{GNT1}, we denote $H\triangleq\sqrt{-\Delta(\kappa^{-1}-\Delta)}$, $U\triangleq\sqrt{\frac{-\Delta}{\kappa^{-1}-\Delta}}$ and define $z\triangleq U^{-1}\nabla a+i\cQ u$, then (\ref{spectrum}) could be rewritten as
$$i\partial_{t}z-2i\Delta z-\sqrt{\kappa}Hz=-2iH\nabla a+iU^{-1}\nabla f-g.$$
Then Duhamel formula yields
\begin{eqnarray}\label{Duhamel}
z(t)=e^{i\sqrt{\kappa} Ht}e^{\Delta t}z_{0}+\int^{t}_{0}e^{i\sqrt{\kappa} H(t-s)}e^{\Delta (t-s)}(-2iH \nabla a+iU^{-1}\nabla f-g)(s)ds.
\end{eqnarray}
where $z_{0}=U^{-1}\nabla a_{0}+i\cQ u_{0}$. The above equation defines a dissipative-dispersive semi-group $e^{i\sqrt{\kappa} Ht}e^{\Delta t}$ and we would establish corresponding Strichartz type estimate containing smooth effect under dissipation. Let us first state the following Proposition concern the corresponding semi-group estimate under localization:
\begin{prop}\label{dispersive estimates}
Set $\kappa>0$. Let $p\in[2,\infty]$ and $1=\frac{1}{p}+\frac{1}{p'}$. For any distribution $f$, there holds for $t>0$
\begin{eqnarray}\label{infty}
\|e^{i\sqrt{\kappa} Ht}e^{\Delta t}f_{j}\|_{L^p}\lesssim  (\sqrt{\kappa} t)^{\frac{d}{p}-\frac{d}{2}}e^{-2^{2j}t}\|f_{j}\|_{L^{p'}}.
\end{eqnarray}
\end{prop}

Above Proposition is obtained by dispersive estimates in terms of parameter $\kappa$, see Theorem 2.1 in \cite{GNT1,GNT2} and heat kernel estimates under Fourier localization. Since it is
fundamental for our analysis, a sketch of it is given in Appendix A. Based on above proposition, we state the following Strichartz type estimates in Besov space which shall be frequently applied in this paper:
\begin{lem}\label{Strichartz estimates}
Set $\kappa>0$, $q,r\in[1,\infty]$, $p\in[2,\infty]$ and $s\in \mathbb{R}$. Let $(r,p)$ satisfies the condition:
\begin{equation}\label{admissible}
\left\{
\begin{array}{l}\frac{d}{2}-\frac{d}{p}\leq\frac{2}{r}\leq\frac{d}{p}-\frac{d}{2}+2,\quad d\geq3\\ [1mm]
1-\frac{2}{p}\leq\frac{2}{r}<\frac{2}{p}+1,\quad \quad \quad d=2.
 \end{array} \right.
\end{equation}
For any $k$ satisfies $k=\frac{2}{r}+\frac{d}{p}-\frac{d}{2}$, there holds
\begin{eqnarray}\label{infty}
\|e^{i\sqrt{\kappa} Ht}e^{\Delta t}u\|_{\tilde{L}^{r}_{T}(\dot{B}^{s+k}_{p,q})}\lesssim  \kappa^{\frac{1}{4}(k-\frac{2}{r})}\|u\|_{\dot{B}^{s}_{2,q}};
\end{eqnarray}
\begin{eqnarray}\label{infty2}
\big\|\int^{t}_{0}e^{i\sqrt{\kappa} H(t-s)}e^{\Delta (t-s)}f(s)ds\big\|_{\tilde{L}^{r}_{T}(\dot{B}^{s+k}_{p,q})}
\lesssim  \kappa^{\frac{1}{4}(k-\frac{2}{r})}\|f\|_{\tilde{L}^1_{T}(\dot{B}^{s}_{2,q})}.
\end{eqnarray}
\end{lem}

\begin{proof}
In the case $p=2$, the optimal smoothing effect of heat kernels, see for example \cite{D}, allows us to have
\begin{eqnarray}
\|e^{i\sqrt{\kappa} Ht}e^{\Delta t}u\|_{\tilde{L}^{r}_{T}(\dot{B}^{s+\frac{2}{r}}_{2,q})}\lesssim\|u\|_{\dot{B}^{s}_{2,q}};
\end{eqnarray}
\begin{eqnarray}
\|\int^{t}_{0}e^{i\sqrt{\kappa} H(t-s)}e^{\Delta (t-s)}f(s)ds\|_{\tilde{L}^{r}_{T}(\dot{B}^{s+\frac{2}{r}}_{2,q})}\lesssim\|u\|_{\tilde{L}^1_{T}(\dot{B}^{s}_{2,q})}.
\end{eqnarray}
As for case $p>2$, we start from $d\leq3$. We begin from $\frac{d}{2}-\frac{d}{p}=\frac{2}{r}$, which corresponds to classical Schr$\ddot{o}$dinger admissible pair. By $TT^*$ argument, it is enough to prove that
$$\|J f_{j}\|_{L^r L^p}\lesssim \kappa^{-\frac{1}{2r}} \|f_{j}\|_{L^{r'} L^{p'}}$$
where
$$Jf_{j}\triangleq\int^{t}_{0} e^{i\sqrt{\kappa} H(s-\sigma)}e^{\Delta(2t-s-\sigma)}f_{j}(\sigma)d\sigma.$$
By Proposition \ref{dispersive estimates}, there holds for $\frac{d}{2}-\frac{d}{p}=\frac{2}{r}$ such that
\begin{eqnarray}
\|Jf_{j}\|_{L^r L^p}&\lesssim&  \big\||\sqrt{\kappa} (t-s)|^{\frac{d}{p}-\frac{d}{2}}e^{-2^{2j}(t-s)}\|f_{j}\|_{L^{p'}}\big\|_{L^r}\\
\nonumber&\lesssim& \kappa^{\frac{d}{2p}-\frac{d}{4}}\big\| t^{\frac{d}{p}-\frac{d}{2}}\|f_{j}\|_{L^{p'}}\big\|_{L^r}.
\end{eqnarray}
Since that $\frac{2}{r}+\frac{d}{p}=\frac{d}{2}$, by Hardy-Littlewood-Sobolev inequality, we have
\begin{eqnarray}
\|Jf_{j}\|_{L^r L^p}\lesssim \kappa^{-\frac{1}{r}}\|f_{j}\|_{L^{r'}L^{p'}}.
\end{eqnarray}
Then a duality argument allows us to have (\ref{infty2}). For case $\frac{d}{2}-\frac{d}{p}<\frac{2}{r}$, we shall utilize the following interpolation
$$\tilde{L}^{r}_{T}(\dot{B}^{s+k}_{p,q})=\Big(\tilde{L}^{r_{1}}_{T}(\dot{B}^{s}_{p_{1},q}),\tilde{L}^{r_{2}}_{T}(\dot{B}^{s+\frac{2}{r_{2}}}_{2,q})\Big)_{\theta}$$
provided
$$\frac{1}{r}=\frac{\theta}{r_{1}}+\frac{1-\theta}{r_{2}};\quad k=\frac{2(1-\theta)}{r_{2}};\quad \frac{1}{p}=\frac{\theta}{p_{1}}+\frac{1-\theta}{2}.$$
Hence, keep in mind $r_{1}\geq2$ and $r_{2}\geq1$, there holds
$$\frac{2}{r}-\frac{d}{p}+\frac{d}{2}-1\leq\frac{4\theta}{r_{1}}+\frac{2(1-\theta)}{r_{2}}-2\leq0,$$
then the interpolation leads us to for $\frac{2}{r}\leq\frac{d}{p}-\frac{d}{2}+2$
\begin{eqnarray}
\|e^{i\sqrt{\kappa} Ht}e^{\Delta t}u\|_{\tilde{L}^{r}_{T}(\dot{B}^{s+k}_{p,q})}\lesssim  \kappa^{-\frac{\theta}{2r_{1}}}\|u\|_{\dot{B}^{s}_{2,q}}
\lesssim  \kappa^{\frac{1}{4}(k-\frac{2}{r})}\|u\|_{\dot{B}^{s}_{2,q}}.
\end{eqnarray}
Similarly we could get to (\ref{infty2}).

As for case $d=2$, we follow the similar calculations as $d\geq3$ but notice critical admissible pair $(r,p,k)=(2,\infty,0)$ is excluded.
\end{proof}

According to Lemma \ref{Strichartz estimates}, we have for $(\delta,p)$ satisfies (\ref{pair}) such that
\begin{eqnarray}
\|z\|_{\tilde{L}^{2}_{T}(\dot{B}^{\frac{d}{p}}_{p,1})}\lesssim \kappa^{-\delta}\big(\|z_{0}\|_{\dot{B}^{\frac{d}{2}-1}_{2,1}}
+\|H\nabla a\|_{L^{1}_{T}(\dot{B}^{\frac{d}{2}-1}_{2,1})}+\|(U^{-1}\nabla f, g)\|_{L^{1}_{T}(\dot{B}^{\frac{d}{2}-1}_{2,1})}\big).
\end{eqnarray}
Notice that
$$\|H\nabla a\|_{L^{1}_{T}(\dot{B}^{\frac{d}{2}-1}_{2,1})}\lesssim\|(\frac{1}{\sqrt{\kappa}} a,\nabla a)\|_{L^{1}_{T}(\dot{B}^{\frac{d}{2}+1}_{2,1})};$$
$$\|U^{-1}\nabla f\|_{L^{1}_{T}(\dot{B}^{\frac{d}{2}-1}_{2,1})}\lesssim\|(\frac{1}{\sqrt{\kappa}}f,\nabla f)\|_{L^{1}_{T}(\dot{B}^{\frac{d}{2}-1}_{2,1})},$$
hence, we could conclude with
\begin{eqnarray}\label{dispersion}
\|z\|_{\tilde{L}^{2}_{T}(\dot{B}^{\frac{d}{p}}_{p,1})}\lesssim \kappa^{-\delta}\big(\|z_{0}\|_{\dot{B}^{\frac{d}{2}-1}_{2,1}}
+\mathcal{E}_{T}+\|(\frac{1}{\sqrt{\kappa}}f,\nabla f, g)\|_{L^{1}_{T}(\dot{B}^{\frac{d}{2}-1}_{2,1})}\big).
\end{eqnarray}

{\it Step3: \underline{Estimates for incompressible part $\cP u$}}

The incompressible estimates follow similar calculations in \cite{D2}. In fact, by Leray projector, the system (INS) could be rewritten as
\begin{equation}
\left\{
\begin{array}{l}\partial_{t}v-\Delta v=\mathcal{P}(v\cdot\nabla v),\\ [1mm]
v|_{t=0}=\mathcal{P}u_{0},
 \end{array} \right.\label{incompressible0}
\end{equation}
while the incompressible part of $\cP u$ fulfills
$$\partial_{t}\mathcal{P}u-\Delta\mathcal{P}u=\mathcal{P}(g_{1}+g_{2})$$
where we take advantages of the fact that for any potential function $w=\nabla f$, there holds $\cP w=0$. Now define $\tilde{v}=\mathcal{P}u-v$, the error of the incompressible part of fluids satisfies
\begin{eqnarray}
\partial_{t}\tilde{v}-\Delta\tilde{v}=\mathcal{P}(\tilde{v}\cdot\nabla\tilde{v})+\mathcal{P}(v\cdot\nabla\tilde{v})+\mathcal{P}(\tilde{v}\cdot\nabla v)+\mathcal{P}\bar{g}
\end{eqnarray}
where $\bar{g}=\cQ u\cdot\nabla\cQ u+\cQ u\cdot\nabla\cP u+\cP u\cdot\nabla\cQ u+g_{2}$.
Then by localization and do inner product with $\tilde{v}_{j}$, we reach
\begin{multline*}
\frac{1}{2}\frac{d}{dt}\|\tilde{v}_{j}\|^{2}_{L^{2}}+c2^{2j}\|\tilde{v}_{j}\|^{2}_{L^{2}}\lesssim\int_{\mathbb{R}^{d}}[\cP\ddj,v\cdot\nabla]\tilde{v}\tilde{v}_{j}dx
+\int_{\mathbb{R}^{d}}v\cdot\nabla \tilde{v}_{j}\tilde{v}_{j}dx\\
+\|\tilde{v}_{j}\|_{L^{2}}(\|\mathcal{P}(v\cdot\nabla\tilde{v})\|_{L^2}+\|\mathcal{P}(\tilde{v}\cdot\nabla\tilde{v})\|_{L^2}+\|\mathcal{P}\bar{g}_{j}\|_{L^2}).
\end{multline*}
Taking advantage of the commutator estimates
\begin{eqnarray*}
\|[\cP\ddj,v\cdot\nabla]\tilde{v}\|_{L^{2}}\lesssim\|\nabla v\|_{\dot{B}^{\frac{d}{2}+1}_{2,1}}\|\tilde{v}_{j}\|_{L^2},
\end{eqnarray*}
and notice
\begin{eqnarray*}
\int_{\mathbb{R}^{d}}v\cdot\nabla \tilde{v}_{j}\tilde{v}_{j}dx\lesssim\|\nabla v\|_{\dot{B}^{\frac{d}{2}+1}_{2,1}}\|\tilde{v}_{j}\|_{L^2},
\end{eqnarray*}
then by maximal regularity of heat equation, there holds
\begin{eqnarray*}
\|\tilde{v}\|_{L^{\infty}_{T}(\dot{B}^{\frac{d}{2}-1}_{2,1})\cap L^1_{T}(\dot{B}^{\frac{d}{2}+1}_{2,1})}
\lesssim\int^{T}_{0}\|v\|_{\dot{B}^{\frac{d}{2}+1}_{2,1}}\|\tilde{v}\|_{\dot{B}^{\frac{d}{2}-1}_{2,1}}ds
+\|\mathcal{P}(\tilde{v}\cdot\nabla\tilde{v})+\mathcal{P}\bar{g}\|_{L^{1}_{T}(\dot{B}^{\frac{d}{2}-1}_{2,1})}.
\end{eqnarray*}
Hence, the Gronwall inequality leads to
\begin{eqnarray}\label{dissipation2}
\|\tilde{v}\|_{L^{\infty}_{T}(\dot{B}^{\frac{d}{2}-1}_{2,1})\cap L^1_{T}(\dot{B}^{\frac{d}{2}+1}_{2,1})}
\lesssim e^{\tilde{C}\mathcal{V}_{T}}\|\mathcal{P}(\tilde{v}\cdot\nabla\tilde{v})+\mathcal{P}\bar{g}\|_{L^{1}_{T}(\dot{B}^{\frac{d}{2}-1}_{2,1})}
\end{eqnarray}
where $\tilde{C}$ is a positive constant. Now by combining (\ref{dissipation}), (\ref{dispersion}) with (\ref{dissipation2}), one could always find a positive $\bar{\kappa}$ such that for all $\kappa\geq\bar{\kappa}$, we arrive at the linear estimates such that
\begin{multline}\label{linear1}
\mathcal{D}_{T}+\mathcal{E}_{T}\leq C\big( \|(\frac{1}{\sqrt{\kappa}}a_{0},\nabla a_{0}, \cQ u_{0})\|_{\dot{B}^{\frac{d}{2}-1}_{2,1}}
+\|(\frac{1}{\sqrt{\kappa}}f,\cQ g-\cQ g_{4})\|_{L^{1}_{T}(\dot{B}^{\frac{d}{2}-1}_{2,1})}\\
+\|u\cdot\nabla a\|_{L^{1}_{T}(\dot{B}^{\frac{d}{2}}_{2,1})}+\sqrt{\kappa}\int^{T}_{0}\|[\ddj,\tilde{\psi}(\kappa^{-\frac{1}{2}}a)](\mathrm{div}u+\Delta a)\|_{\dot{B}^{\frac{d}{2}}_{2,1}}ds\big)
\end{multline}
and
\begin{eqnarray}\label{linear2}
\mathcal{W}_{T}\leq Ce^{\tilde{C}\mathcal{V}_{T}}\|\mathcal{P}(\tilde{v}\cdot\nabla\tilde{v})+\mathcal{P}\bar{g}\|_{L^{1}_{T}(\dot{B}^{\frac{d}{2}-1}_{2,1})}.
\end{eqnarray}

{\it Step4: \underline{Nonlinear estimates}}

So what left is to give nonlinear estimates. We begin with nonlinearities in (\ref{linear1}). Bounding $\|\frac{1}{\sqrt{\kappa}}f\|_{L^{1}_{T}(\dot{B}^{\frac{d}{2}-1}_{2,1})}$ is obtained by
\begin{eqnarray*}
\|\tilde{\psi}(\kappa^{-\frac{1}{2}}a)\mathrm{div}u\|_{L^{1}_{T}(\dot{B}^{\frac{d}{2}}_{2,1})}
\lesssim\frac{1}{\sqrt{\kappa}}\|a\|_{L^{\infty}_{T}L^\infty}\|\mathrm{div}\cQ u\|_{L^{1}_{T}(\dot{B}^{\frac{d}{2}}_{2,1})}\lesssim\kappa^{-\frac{1}{2}}\mathcal{E}^{2}_{T};
\end{eqnarray*}
\begin{eqnarray*}
\frac{1}{\sqrt{\kappa}}\|u\cdot \nabla a\|_{L^{1}_{T}(\dot{B}^{\frac{d}{2}}_{2,1})}\lesssim\frac{1}{\sqrt{\kappa}}\|u\|_{L^{2}_{T}L^\infty}\|\nabla a\|_{L^{2}_{T}(\dot{B}^{\frac{d}{2}}_{2,1})}
\lesssim\kappa^{-\frac{1}{2}}\mathcal{E}_{T}(\mathcal{E}_{T}+\mathcal{W}_{T}+\mathcal{V}_{T}),
\end{eqnarray*}
which leads to
\begin{eqnarray}
\|\frac{1}{\sqrt{\kappa}}f\|_{L^{1}_{T}(\dot{B}^{\frac{d}{2}-1}_{2,1})}\lesssim\kappa^{-\frac{1}{2}}(\mathcal{E}^{2}_{T}+\mathcal{W}^{2}_{T}+\mathcal{V}^{2}_{T}).
\end{eqnarray}
On the other hand, for $\|u\cdot\nabla a\|_{L^{1}_{T}(\dot{B}^{\frac{d}{2}}_{2,1})}$, we immediately have
\begin{eqnarray*}
\|\cQ u\cdot \nabla a\|_{L^{1}_{T}(\dot{B}^{\frac{d}{2}}_{2,1})}\lesssim\|\nabla a\|_{L^{2}_{T}L^\infty}\|\cQ u\|_{L^{2}_{T}(\dot{B}^{\frac{d}{2}}_{2,1})}
+\|\cQ u\|_{L^{2}_{T}L^\infty}\|\nabla a\|_{L^{2}_{T}(\dot{B}^{\frac{d}{2}}_{2,1})}.
\end{eqnarray*}
Since that there holds
\begin{eqnarray}\label{infty}
\|(\nabla a,\cQ u)\|_{L^{2}_{T}L^\infty}\lesssim\|(U^{-1}\mathrm{Re}z,\mathrm{Im}z)\|_{L^{2}_{T}(\dot{B}^{\frac{d}{p}}_{p,1})}\lesssim\kappa^{-\delta}\mathcal{D}_{T},
\end{eqnarray}
provided $p$ satisfies (\ref{pair}), we obtain
\begin{eqnarray}
\|\cQ u\cdot \nabla a\|_{L^{1}_{T}(\dot{B}^{\frac{d}{2}}_{2,1})}\lesssim\kappa^{-\delta}\mathcal{D}_{T}\mathcal{E}_{T}.
\end{eqnarray}
To deal with interaction of $\cP u$, we use the Bony decomposition
$$\cP u\cdot \nabla a=T_{\cP u}\nabla a+R(\cP u,\nabla a)+T_{\nabla a}\cP u.$$
We have the following estimate such that
\begin{eqnarray*}
\|T_{\cP u}\nabla a\|_{L^{1}_{T}(\dot{B}^{\frac{d}{2}}_{2,1})}
\lesssim\|\cP u\|_{L^{4}_{T}(\dot{B}^{\frac{d}{2}-\frac{1}{2}}_{2,1})}\|\nabla a\|_{L^{\frac{4}{3}}_{T}(\dot{B}^{\frac{d}{\tilde{p}}+\frac{1}{2}}_{\tilde{p},1})}
\end{eqnarray*}
provided $\frac{d}{2}-\frac{1}{2}=\frac{d}{\tilde{p}}$. By taking $k=1, r=\frac{4}{3}$ in Lemma \ref{Strichartz estimates}, there holds
$$\|\nabla a\|_{L^{\frac{4}{3}}_{T}(\dot{B}^{\frac{d}{\tilde{p}}+\frac{1}{2}}_{\tilde{p},1})}\lesssim
\|\nabla a\|_{\tilde{L}^{\frac{4}{3}}_{T}(\dot{B}^{\frac{d}{\tilde{p}}+\frac{1}{2}}_{\tilde{p},1})}\lesssim\kappa^{-\frac{\delta}{2}}\mathcal{D}_{T}.$$
Consequently, the Cauchy inequality implicates that
\begin{multline}\label{momoo}
\|T_{\cP u}\nabla a\|_{L^{1}_{T}(\dot{B}^{\frac{d}{2}}_{2,1})}\lesssim\kappa^{-\frac{\delta}{2}}(\mathcal{W}_{T}+\mathcal{V}_{T})(\mathcal{D}_{T}+\mathcal{E}_{T})
\lesssim(\mathcal{W}_{T}+\mathcal{V}_{T})^2+\kappa^{-\delta}(\mathcal{D}_{T}+\mathcal{E}_{T})^{2}.
\end{multline}
On the other hand, there holds
\begin{multline}\label{momo}
\|R(\cP u,\nabla a)+T_{\nabla a}\cP u\|_{L^{1}_{T}(\dot{B}^{\frac{d}{2}}_{2,1})}\lesssim\|\nabla a\|_{L^{2}_{T}L^\infty}\|\cP u\|_{L^{2}_{T}(\dot{B}^{\frac{d}{2}}_{2,1})}
\lesssim\kappa^{-\delta}\mathcal{D}_{T}(\mathcal{W}_{T}+\mathcal{V}_{T}).
\end{multline}
Hence, combining (\ref{momoo}) with (\ref{momo}) yields
\begin{eqnarray}\label{f}
\|u\cdot \nabla a\|_{L^{1}_{T}(\dot{B}^{\frac{d}{2}}_{2,1})}\lesssim(\mathcal{W}_{T}+\mathcal{V}_{T})^2+\kappa^{-\frac{1}{4}}(\mathcal{D}_{T}+\mathcal{E}_{T})^{2}.
\end{eqnarray}

Next we turn to $g$. For the convection term $g_{1}$, we can follow similar steps as (\ref{f}) which yields
\begin{eqnarray}
\|\mathcal{Q}u\cdot\nabla \mathcal{Q}u\|_{L^{1}_{T}(\dot{B}^{\frac{d}{2}-1}_{2,1})}
\lesssim\|\mathcal{Q}u\|_{L^{2}_{T}L^\infty}\|\mathcal{Q}u\|_{L^{2}_{T}(\dot{B}^{\frac{d}{2}}_{2,1})}\lesssim\kappa^{-\delta}\mathcal{D}_{T}\mathcal{E}_{T};
\end{eqnarray}
\begin{eqnarray}
\|\mathcal{P}u\cdot\nabla \mathcal{Q}u+\mathcal{Q}u\cdot\nabla \mathcal{P}u\|_{L^{1}_{T}(\dot{B}^{\frac{d}{2}-1}_{2,1})}
&\lesssim&\|\mathcal{P}u\mathcal{Q}u\|_{L^{1}_{T}(\dot{B}^{\frac{d}{2}}_{2,1})}\\
\nonumber&\lesssim&(\mathcal{W}_{T}+\mathcal{V}_{T})^2+\kappa^{-\delta}(\mathcal{D}_{T}+\mathcal{E}_{T})^{2};
\end{eqnarray}
\begin{eqnarray}
\|\mathcal{P}u\cdot\nabla\mathcal{P}u\|_{L^{1}_{T}(\dot{B}^{\frac{d}{2}-1}_{2,1})}
\lesssim\|\mathcal{P}u\|_{L^{\infty}_{T}(\dot{B}^{\frac{d}{2}-1}_{2,1})}\|\mathcal{P}u\|_{L^{1}_{T}(\dot{B}^{\frac{d}{2}+1}_{2,1})}\lesssim(\mathcal{V}_{T}+\mathcal{W}_{T})^{2}.
\end{eqnarray}

For the viscosity term $g_{2}$, , we shall estimates the term $\mathrm{div}(2\tilde{\mu}(\kappa^{-\frac{1}{2}}a) D(u))$ where
\begin{eqnarray*}
\|\big(\mathrm{div}(2\tilde{\mu}(\kappa^{-\frac{1}{2}}a) D(u))\big)\|_{L^{1}_{T}(\dot{B}^{\frac{d}{2}-1}_{2,1})}\lesssim\kappa^{-\frac{1}{2}}
\|a\|_{L^{\infty}_{T}(\dot{B}^{\frac{d}{2}}_{2,1})}\|D(u)\|_{L^{1}_{T}(\dot{B}^{\frac{d}{2}}_{2,1})}.
\end{eqnarray*}

Similar calculations on $\big(1+Q(\kappa^{-\frac{1}{2}}a)\big)\nabla(\tilde{\lambda}(\kappa^{-\frac{1}{2}}a)\mathrm{div}u)$ shares
\begin{eqnarray}\label{momooo}
\|\cQ g_{2}\|_{L^{1}_{T}(\dot{B}^{\frac{d}{2}-1}_{2,1})}&\lesssim&\kappa^{-\frac{1}{2}}(1+\mathcal{E}_{T})\mathcal{E}_{T}(\mathcal{E}_{T}+\mathcal{V}_{T}+\mathcal{W}_{T}).
\end{eqnarray}

Next we turn to the pressure term $g_{3}$, we have
\begin{eqnarray*}
\|\tilde{G}(\kappa^{-\frac{1}{2}}a)\nabla a\|_{L^{1}_{T}(\dot{B}^{\frac{d}{2}-1}_{2,1})}
\lesssim\|\tilde{G}(\kappa^{-\frac{1}{2}}a)\|_{L^{2}_{T}L^\infty}\|a\|_{L^{2}_{T}(\dot{B}^{\frac{d}{2}}_{2,1})}
\lesssim\|a\|_{L^{2}_{T}L^\infty}\|\frac{a}{\sqrt{\kappa}}\|_{L^{2}_{T}(\dot{B}^{\frac{d}{2}}_{2,1})}.
\end{eqnarray*}
Since that
\begin{eqnarray}\label{infty0}
\kappa^{-\frac{1}{2}}\|a\|_{L^{2}_{T}L^\infty}\lesssim\kappa^{-\frac{1}{2}}\|U^{-1}\mathrm{Re}z\|_{L^{2}_{T}(\dot{B}^{\frac{d}{p}-1}_{p,1})}
\lesssim\|z\|_{L^{2}_{T}(\dot{B}^{\frac{d}{p}}_{p,1})}\lesssim\kappa^{-\delta}\mathcal{D}_{T},
\end{eqnarray}
we arrive at
\begin{eqnarray}
\|\kappa^{-\frac{1}{2}}\tilde{G}(\kappa^{-\frac{1}{2}}a)\nabla a\|_{L^{1}_{T}(\dot{B}^{\frac{d}{2}-1}_{2,1})}\lesssim\kappa^{-\delta}\mathcal{D}_{T}\mathcal{E}_{T}.
\end{eqnarray}
Finally we bound Korteweg nonlinearities $g_{5}$, in fact, it holds
\begin{eqnarray}
\|\nabla(|\nabla a\otimes\nabla a|^2)\|_{L^{1}_{T}(\dot{B}^{\frac{d}{2}-1}_{2,1})}\lesssim\|\nabla a\|_{L^{2}_{T}L^\infty}
\|\nabla a\|_{L^{2}_{T}(\dot{B}^{\frac{d}{2}}_{2,1})}\lesssim\kappa^{-\delta}\mathcal{D}_{T}\mathcal{E}_{T}.
\end{eqnarray}

By commutator estimates we have
\begin{multline}
\sqrt{\kappa}\int^{T}_{0}\|[\ddj,\tilde{\psi}(\kappa^{-\frac{1}{2}}a)]\mathrm{div}u\|_{\dot{B}^{\frac{d}{2}}_{2,1}}ds+\sqrt{\kappa}\int^{T}_{0}\|[\ddj,\tilde{\psi}(\kappa^{-\frac{1}{2}}a)]\Delta a\|_{\dot{B}^{\frac{d}{2}}_{2,1}}ds\\
\lesssim\|\nabla a\|_{L^{2}_{T}L^\infty}\|(\nabla a,\cQ u)\|_{L^{2}_{T}(\dot{B}^{\frac{d}{2}}_{2,1})}\lesssim\kappa^{-\delta}\mathcal{D}_{T}\mathcal{E}_{T}.
\end{multline}
Therefore, combining linear estimates (\ref{linear1}) and nonlinear estimates above, we finish the proof of (\ref{1}). As for nonlinear estimates in (\ref{linear2}), we have
\begin{eqnarray}
\|\mathcal{P}(\tilde{v}\cdot\nabla\tilde{v})\|_{L^{1}_{T}(\dot{B}^{\frac{d}{2}-1}_{2,1})}
\lesssim\|\tilde{v}\|_{L^{\infty}_{T}(\dot{B}^{\frac{d}{2}-1}_{2,1})}\|\nabla\tilde{v}\|_{L^{1}_{T}(\dot{B}^{\frac{d}{2}}_{2,1})}
\lesssim\mathcal{W}^{2}_{T}.
\end{eqnarray}
On the other hand, repeating calculations in (\ref{momo})-(\ref{momoo}) and (\ref{momooo}), we obtain
\begin{eqnarray}
\|\mathcal{P}\bar{g}\|_{L^{1}_{T}(\dot{B}^{\frac{d}{2}-1}_{2,1})}
\lesssim\kappa^{-\frac{\delta}{2}}(1+\mathcal{E}_{T})(\mathcal{D}_{T}+\mathcal{E}_{T}+\mathcal{V}_{T}+\mathcal{W}_{T})^{2}
\end{eqnarray}
which leads to (\ref{2}) and we finish the proof of Proposition \ref{uniform estimates}.
\end{proof}

\subsection{Bootstrap and continuation argument}
Denote $\mathcal{Z}_{0}=\|(\frac{1}{\sqrt{\kappa}}a_{0},\nabla a_{0}, \cQ u_{0})\|_{\dot{B}^{\frac{d}{2}-1}_{2,1}}$, we state the following Lemma concerns uniform bound for $\mathcal{D}_{T},\mathcal{E}_{T},\mathcal{W}_{T}$:
\begin{lem}\label{continuation argument}
Assume $T>0$ to be finite or infinite and $\mathcal{V}_{T}$ is bounded, then there exists a large enough $\kappa_{0}>0$ depends on $\|(a_{0},\nabla a_{0}, \cQ u_{0})\|_{\dot{B}^{\frac{d}{2}-1}_{2,1}}, \mathcal{V}_{T}$ such that
it holds true for all $\kappa\geq\kappa_{0}$
\begin{eqnarray}\label{11}
\mathcal{D}_{T}+\mathcal{E}_{T}\leq2C(\mathcal{Z}_{0}+\mathcal{V}^{2}_{T});
\end{eqnarray}
\begin{eqnarray}\label{22}
\mathcal{W}_{T}\leq2\kappa^{-\frac{\delta}{2}}Ce^{\tilde{C}\mathcal{V}_{T}}(1+2C(\mathcal{Z}_{0}+\mathcal{V}^{2}_{T}))(C(\mathcal{Z}_{0}+\mathcal{V}^{2}_{T})+\mathcal{V}^{2}_{T})^2.
\end{eqnarray}
\end{lem}

\begin{proof}
We denote:
\begin{multline*}
T^*\triangleq\max\Big\{T^*\leq T:\mathcal{D}_{T^*}+\mathcal{E}_{T^*}\leq 2C(\mathcal{Z}_{0}+\mathcal{V}^{2}_{T^*});\\
\mathcal{W}_{T^*}\leq2\kappa^{-\frac{\delta}{2}}Ce^{\tilde{C}\mathcal{V}_{T^*}}(1+\mathcal{Z}_{0}+\mathcal{V}^{2}_{T^*})(\mathcal{Z}_{0}+2\mathcal{V}^{2}_{T^*})^{2}\Big\}.
\end{multline*}
According to Proposition \ref{uniform estimates}, there holds
\begin{eqnarray*}
\mathcal{D}_{T^*}+\mathcal{E}_{T^*}\!\!&\leq&\!\! C\Big(\mathcal{Z}_{0}+\mathcal{V}^{2}_{T^*}+4\kappa^{-2\delta}C^{2}e^{\tilde{C}\mathcal{V}_{T^*}}(1+2C(\mathcal{Z}_{0}+\mathcal{V}^{2}_{T^*}))^2
(C(\mathcal{Z}_{0}+\mathcal{V}^{2}_{T^*})+\mathcal{V}^{2}_{T^*})^4\\
&+&\kappa^{-\delta}(1+2C(\mathcal{Z}_{0}+\mathcal{V}^{2}_{T^*}))(C(\mathcal{Z}_{0}+\mathcal{V}^{2}_{T^*})+\mathcal{V}^{2}_{T^*})^2\Big)
\end{eqnarray*}
while
\begin{eqnarray*}
\mathcal{W}_{T^*}&\leq&4\kappa^{-2\delta}C^{3}e^{3\tilde{C}\mathcal{V}_{T^*}}(1+2C(\mathcal{Z}_{0}+\mathcal{V}^{2}_{T^*}))^2
(C(\mathcal{Z}_{0}+\mathcal{V}^{2}_{T^*})+\mathcal{V}^{2}_{T^*})^4\\
&+&\kappa^{-\delta}Ce^{\tilde{C}\mathcal{V}_{T^*}}(1+2C(\mathcal{Z}_{0}+\mathcal{V}^{2}_{T^*}))(C(\mathcal{Z}_{0}+\mathcal{V}^{2}_{T^*})+\mathcal{V}^{2}_{T^*})^2.
\end{eqnarray*}

By selecting the $\kappa_{0}$ satisfies
\begin{multline*}
4\kappa_{0}^{-2\delta}C^{2}e^{\tilde{C}\mathcal{V}_{T^*}}(1+2C(\mathcal{Z}_{0}+\mathcal{V}^{2}_{T^*}))^2
(C(\mathcal{Z}_{0}+\mathcal{V}^{2}_{T^*})+\mathcal{V}^{2}_{T^*})^4\\
+\kappa_{0}^{-\delta}(1+2C(\mathcal{Z}_{0}+\mathcal{V}^{2}_{T^*}))(C(\mathcal{Z}_{0}+\mathcal{V}^{2}_{T^*})+\mathcal{V}^{2}_{T^*})^2=\frac{1}{2}C(\mathcal{Z}_{0}+\mathcal{V}^{2}_{T^*});
\end{multline*}
\begin{multline*}
4\kappa_{0}^{-\delta}C^{3}e^{3\tilde{C}\mathcal{V}_{T^*}}(1+2C(\mathcal{Z}_{0}+\mathcal{V}^{2}_{T^*}))^2
(C(\mathcal{Z}_{0}+\mathcal{V}^{2}_{T^*})+\mathcal{V}^{2}_{T^*})^4\\
=\frac{1}{2}Ce^{\tilde{C}\mathcal{V}_{T^*}}(1+\mathcal{Z}_{0}+\mathcal{V}^{2}_{T^*})(\mathcal{Z}_{0}+2\mathcal{V}^{2}_{T^*})^{2},
\end{multline*}
we could conclude on $[0,T^*]$ such that for all $\kappa\geq\kappa_{0}$
\begin{eqnarray*}
\mathcal{D}_{T^*}+\mathcal{E}_{T^*}\leq\frac{3}{2}C(\mathcal{Z}_{0}+\mathcal{V}^{2}_{T});
\end{eqnarray*}
\begin{eqnarray*}
\mathcal{W}_{T^*}\leq\frac{3}{2}\kappa^{-\delta}Ce^{\tilde{C}\mathcal{V}_{T^*}}(1+2C(\mathcal{Z}_{0}+\mathcal{V}^{2}_{T^*}))(C(\mathcal{Z}_{0}+\mathcal{V}^{2}_{T^*})+\mathcal{V}^{2}_{T^*})^2
\end{eqnarray*}
which contracts with definition of $T^*$ and we could extend to $T$ beyond $[0,T^*]$.
\end{proof}

Moreover, we introduce the following local well-posedness has been proved in \cite{DD}:
\begin{lem}
There exists a $\eta>0$ such that if $(\nabla a_{0},u_{0})\in\dot{B}^{\frac{d}{2}-1}_{2,1}$ and
$$\|a_{0}\|_{\dot{B}^{\frac{d}{2}}_{2,1}}\leq\eta$$
there exist a $T>0$ such that (\ref{1.1})-(\ref{1.2}) admits a unique solution $(a,u)\in\mathcal{C}_{T}(\dot{B}^{\frac{d}{2}-1}_{2,1})\cap L^1_{T}(\dot{B}^{\frac{d}{2}+1}_{2,1})$ satisfying
\begin{eqnarray}
\|(a,u)\|_{L^\infty_{T}(\dot{B}^{\frac{d}{2}-1}_{2,1})}+\|(a,u)\|_{L^1_{T}(\dot{B}^{\frac{d}{2}+1}_{2,1})}\leq C_{T}.
\end{eqnarray}
\end{lem}

Therefore, by a standard bootstrap argument, we could prove that once $\kappa\geq\kappa_{0}$ large enough, the solution could be extended to $[0,\infty]$ while the incompressible limit is ensured by (\ref{2}).

\section{Decay rates}\setcounter{equation}{0}
In this section, we present the proof of Theorem \ref{thm2.2} where we shall establish decay rates for  arbitrary order derivatives. Our method includes two steps where we first reveal the evolution of regularity under $\sigma$. Then we would establish the decay rates for any order derivatives.

\subsection{Evolution of norm under regularity $-\sigma$}
In this section, we establish uniform bounds of the solution in negative Besov norms. More precisely, it is shown that for any $t>0$
\begin{eqnarray}\label{evo}
\|(\frac{1}{\sqrt{\kappa}}a,\nabla a, u)(t)\|_{{\dot B}^{-\sigma}_{2,\infty}}\leq C_0
\end{eqnarray}
where $C_0>0$ depends on the initial norm $\|(\frac{1}{\sqrt{\kappa}}a_{0},\nabla a_{0}, u_{0})\|_{\dot{B}^{-\sigma}_{2,\infty}}$. The key step is to claim the following lemma concerns the evolution of norm under the negative regularity $-\sigma$.
\begin{lem}\label{lem4.1}
Assume $(a,u)$ to be the solution established in Theorem \ref{thm2.1}. Let $\sigma$ fulfills $1-\frac{d}{2}<\sigma\leq\frac{d}{2}$.
It holds that
\begin{multline}\label{energy esti}
\|(\frac{1}{\sqrt{\kappa}}a,\nabla a,  u)(t)\|_{\dot{B}^{-\sigma}_{2,\infty}} \lesssim\|(\frac{1}{\sqrt{\kappa}}a_{0},\nabla a_{0},  u_{0})\|_{\dot{B}^{-\sigma}_{2,\infty}}\\+\int^{t}_{0}(1+\|a\|_{\dot{B}^{\frac{d}{2}}_{2,1}})\|(\frac{1}{\sqrt{\kappa}}a,\nabla a,  u)\|_{\dot{B}^{\frac{d}{2}+1}_{2,1}}\|(\frac{1}{\sqrt{\kappa}}a,\nabla a,  u)\|_{\dot{B}^{-\sigma}_{2,\infty}}ds.
\end{multline}
\end{lem}

\begin{proof}
Let us first recall (\ref{localization estimates}), integral on time leads us to
\begin{multline}
\|(\frac{1}{\sqrt{\kappa}}a_{j},\nabla a_{j}, \cQ u_{j})\|_{L^2}\lesssim e^{-2^{2j}t}\|(\frac{1}{\sqrt{\kappa}}a_{0,j},\nabla a_{0,j}, \cQ u_{0,j})\|_{L^2}\\
+\int^{t}_{0}e^{-2^{2j}(t-s)}\|(\frac{1}{\sqrt{\kappa}}f_{j},\cQ g_{j}-\cQ g_{4,j})\|_{L^{2}}+2^{j}\|\ddj(u\cdot\nabla a)\|_{L^{2}}ds\\
+\sqrt{\kappa}\int^{t}_{0}e^{-2^{2j}(t-s)}2^{j}\big(\|[\ddj,\tilde{\psi}(\kappa^{-\frac{1}{2}}a)](\mathrm{div}u+\Delta a)\|_{L^2}ds.
\end{multline}
Hence, taking $\dot{B}^{-\sigma}_{2,\infty}$ norm yields
\begin{multline}
\|(\frac{1}{\sqrt{\kappa}}a,\nabla a, \cQ u)\|_{\dot{B}^{-\sigma}_{2,\infty}}\lesssim \|(\frac{1}{\sqrt{\kappa}}a_{0},\nabla a_{0}, \cQ u_{0})\|_{\dot{B}^{-\sigma}_{2,\infty}}
+\int^{t}_{0}\|(\frac{1}{\sqrt{\kappa}}f,\cQ g-\cQ g_{4})\|_{\dot{B}^{-\sigma}_{2,\infty}}ds\\
+\int^{t}_{0}\|u\cdot\nabla a\|_{\dot{B}^{-\sigma+1}_{2,\infty}}ds
+\sqrt{\kappa}\int^{t}_{0}\|[\ddj,\tilde{\psi}(\kappa^{-\frac{1}{2}}a)](\mathrm{div}u+\Delta a)\|_{\dot{B}^{-\sigma+1}_{2,\infty}}ds
\end{multline}
while for the incompressible part, we do the similar calculations and reach
\begin{eqnarray}
\|\cP u\|_{\dot{B}^{-\sigma}_{2,\infty}}\lesssim \|\cP u_{0}\|_{\dot{B}^{-\sigma}_{2,\infty}}+\int^{t}_{0}\|\cP g\|_{\dot{B}^{-\sigma}_{2,\infty}}ds.
\end{eqnarray}

For nonlinear terms, we start with $u\cdot\nabla a$, indeed, we have
\begin{eqnarray}
\|u\cdot\nabla a\|_{\dot{B}^{-\sigma+1}_{2,\infty}}\lesssim\|u\|_{\dot{B}^{-\sigma}_{2,\infty}}\|\nabla a\|_{\dot{B}^{\frac{d}{2}+1}_{2,1}}
+\|\nabla a\|_{\dot{B}^{-\sigma}_{2,\infty}}\|u\|_{\dot{B}^{\frac{d}{2}+1}_{2,1}}
\end{eqnarray}
provided $1-\frac{d}{2}<\sigma\leq\frac{d}{2}$. Similarly we could estimate terms $u\cdot\nabla u, \nabla(|\nabla a|^{2}),\kappa^{-\frac{1}{2}}\tilde{G}(\kappa^{-\frac{1}{2}}a)\nabla a$. For commutators, we have
\begin{eqnarray*}
\|[\ddj,\tilde{\psi}(\kappa^{-\frac{1}{2}}a)](\mathrm{div}u+\Delta a)\|_{\dot{B}^{-\sigma+1}_{2,\infty}}&\lesssim&\kappa^{-\frac{1}{2}}\|\nabla a\|_{\dot{B}^{\frac{d}{2}}_{2,1}}
\|(u+\nabla a)\|_{\dot{B}^{-\sigma+1}_{2,\infty}}.
\end{eqnarray*}
Now taking $\theta=\frac{1}{\sigma+\frac{d}{2}+1}$, interpolation allows us to have
$$\|\nabla a\|_{\dot{B}^{\frac{d}{2}}_{2,1}}\leq\|\nabla a\|^{\theta}_{\dot{B}^{-\sigma}_{2,\infty}}\|\nabla a\|^{1-\theta}_{\dot{B}^{\frac{d}{2}+1}_{2,1}}$$
and
$$\|(u+\nabla a)\|_{\dot{B}^{-\sigma+1}_{2,\infty}}\leq\|(\nabla a,u)\|^{1-\theta}_{\dot{B}^{-\sigma}_{2,\infty}}\|(\nabla a,u)\|^{\theta}_{\dot{B}^{\frac{d}{2}+1}_{2,1}},$$
thus we find
\begin{eqnarray*}
\sqrt{\kappa}\|[\ddj,\tilde{\psi}(\kappa^{-\frac{1}{2}}a)](\mathrm{div}u+\Delta a)\|_{\dot{B}^{-\sigma+1}_{2,\infty}}&\lesssim&\|(\nabla a,u)\|_{\dot{B}^{-\sigma}_{2,\infty}}\|(\nabla a,u)\|_{\dot{B}^{\frac{d}{2}+1}_{2,1}}.
\end{eqnarray*}
For the quasi-linear ones, we take a look on $\tilde{\lambda}(\kappa^{-\frac{1}{2}}a)\Delta u$, we actually have
\begin{eqnarray*}
\|\tilde{\lambda}(\kappa^{-\frac{1}{2}}a)\Delta u\|_{\dot{B}^{-\sigma}_{2,\infty}}&\lesssim&\kappa^{-\frac{1}{2}}\|a\|_{\dot{B}^{\frac{d}{2}}_{2,1}}
\|\Delta u\|_{\dot{B}^{-\sigma}_{2,\infty}}\\
&\lesssim&\kappa^{-\frac{1}{2}}\|\nabla a\|^{\tilde{\theta}}_{\dot{B}^{-\sigma}_{2,\infty}}\|\nabla a\|^{1-\tilde{\theta}}_{\dot{B}^{\frac{d}{2}+1}_{2,1}}
\| u\|^{1-\tilde{\theta}}_{\dot{B}^{-\sigma}_{2,\infty}}\| u\|^{\tilde{\theta}}_{\dot{B}^{\frac{d}{2}+1}_{2,1}}
\end{eqnarray*}
with $\tilde{\theta}=\frac{2}{\sigma+\frac{d}{2}+1}$. Again, the Cauchy inequality allows us to arrive at
\begin{eqnarray*}
\|\tilde{\lambda}(\kappa^{-\frac{1}{2}}a)\Delta u\|_{\dot{B}^{-\sigma}_{2,\infty}}\lesssim\kappa^{-\frac{1}{2}}\|(\nabla a,u)\|_{\dot{B}^{-\sigma}_{2,\infty}}\|(\nabla a,u)\|_{\dot{B}^{\frac{d}{2}+1}_{2,1}}.
\end{eqnarray*}
Other terms follow similar calculations and hence, Lemma \ref{lem4.1} is proved.
\end{proof}

Now Theorem \ref{thm2.1} ensures that
\begin{multline}\int^{t}_{0}(1+\|a\|_{\dot{B}^{\frac{d}{2}}_{2,1}})\|(\frac{1}{\sqrt{\kappa}}a,\nabla a,  u)\|_{\dot{B}^{\frac{d}{2}+1}_{2,1}}ds\\
\leq(1+\|a\|_{L^\infty_{T}(\dot{B}^{\frac{d}{2}}_{2,1})})\|(\frac{1}{\sqrt{\kappa}}a,\nabla a,  u)\|_{L^1_{T}(\dot{B}^{\frac{d}{2}+1}_{2,1})}\leq C,\end{multline}
the Gronwall inequality allows us to have (\ref{evo}).

\subsection{Decay estimates of arbitrary order derivatives}
In this subsection, we would establish decay rates for derivatives. Let us first introduce the following signal:
\begin{eqnarray}\label{X}
\mathcal{X}_{t}\triangleq\sup_{t\in[0,T]}t^{\frac{|\alpha|}{2}}\|D^{\alpha}(\frac{1}{\sqrt{\kappa}}a,\nabla a, u)\|_{\dot{B}^{-\sigma}_{2,\infty}}
\end{eqnarray}
where $\alpha$ is the index fulfills $|\alpha|>0$. Moreover, we shall define $\mathcal{Y}_{t}$ to present the decay functional space for solutions of (INS):
\begin{eqnarray}\label{Y}
\mathcal{Y}_{t}\triangleq\sup_{t\in[0,T]}t^{\frac{|\alpha|}{2}}\|D^{\alpha}v\|_{\dot{B}^{-\sigma}_{2,\infty}}.
\end{eqnarray}

Our main purpose is to establish the following Lemma:
\begin{prop}\label{high order estimates}
Let $d\geq2$. Suppose that $(a,u)$ is the solution of system (\ref{1.1})-(\ref{1.2}). Then the following inequality holds true for $|\alpha|\geq\sigma$ and $t\geq1$ such that
\begin{multline}\label{3}
\mathcal{X}_{t}\lesssim \|(\frac{1}{\sqrt{\kappa}}a_{0},\nabla a_{0}, u_{0})\|_{\dot{B}^{-\sigma}_{2,\infty}}
+\|(\frac{1}{\sqrt{\kappa}}a,\nabla a, u)\|_{L^{\infty}_{t}(\dot{B}^{-\sigma}_{2,\infty})}(\mathcal{E}_{t}+\mathcal{W}_{t}+\mathcal{V}_{t})\\
+(\mathcal{V}_{t}+\mathcal{D}_{t})\mathcal{Y}_{t}+(\mathcal{W}_{t}
+\kappa^{-\frac{\delta}{\alpha-\sigma}}t^{-\epsilon}\mathcal{V}_{t}+\kappa^{-\delta}\mathcal{D}_{t})\mathcal{X}_{t}
\end{multline}
where $\epsilon=\frac{d}{2}-1-\sigma$.
\end{prop}

\begin{proof}
Before the detailed calculations, for clarify, we write  $m$ and $h$ to represent
$$m=(\frac{1}{\sqrt{\kappa}}a,\nabla a, \cQ u) \quad\mathrm{and}\quad h=\frac{1}{\sqrt{\kappa}}f+\cQ g-\cQ g_{4}+\nabla(u\cdot\nabla a).$$
Applying derivatives $D^\alpha$ on the variable and repeating the energy under localization in (\ref{energy1}) yields
\begin{multline*}
\frac{d}{dt}\|D^{\alpha}m_{j}\|_{L^2}+2^{2j}\|D^{\alpha}m_{j}\|_{L^2}\\
\lesssim\|D^{\alpha}h_{j}\|_{L^{2}}-\sqrt{\kappa}2^{j}\|[D^{\alpha}\ddj,\tilde{\psi}(\kappa^{-\frac{1}{2}}a)](\Delta a+\mathrm{div}\mathcal{Q}u)\|_{L^{2}}.
\end{multline*}
Then taking advantages of the Gronwall inequality and notice the following estimate
\begin{eqnarray*}
e^{-2^{2j}t}\|D^{\alpha}(\frac{1}{\sqrt{\kappa}}a_{0,j},\nabla a_{0,j}, \cQ u_{0,j})\|_{L^2}
&\lesssim& e^{-2^{2j}t}2^{|\alpha| j}\|(\frac{1}{\sqrt{\kappa}}a_{0,j},\nabla a_{0,j}, \cQ u_{0,j})\|_{L^2}\\
&\lesssim& e^{-c2^{2j}t}t^{-\frac{|\alpha|}{2}}\|(\frac{1}{\sqrt{\kappa}}a_{0,j},\nabla a_{0,j}, \cQ u_{0,j})\|_{L^2}
\end{eqnarray*}
where we use $x^{\eta}e^{-x}\leq e^{-cx}$ for $\eta>0$, we immediately have
\begin{multline}
\|D^{\alpha}m_{j}\|_{L^2}\lesssim t^{-\frac{|\alpha|}{2}}\|m_{0,j}\|_{L^2}+\int^{t}_{0}e^{-2^{2j}(t-s)}\|D^{\alpha}h_{j}\|_{L^{2}}ds\\
+\sqrt{\kappa}\int^{t}_{0}e^{-2^{2j}(t-s)}2^{j}\|[D^{\alpha}\ddj,\tilde{\psi}(\kappa^{-\frac{1}{2}}a)](\Delta a+\mathrm{div}\mathcal{Q}u)\|_{L^{2}}ds.
\end{multline}
Then after taking $\dot{B}^{-\sigma}_{2,\infty}$ norm, the following inequality holds
\begin{multline}\label{decay1}
\|D^{\alpha}m\|_{\dot{B}^{-\sigma}_{2,\infty}}\lesssim t^{-\frac{|\alpha|}{2}}\|m_{0}\|_{\dot{B}^{-\sigma}_{2,\infty}}
+\sup_{j}2^{-\sigma j}\int^{t}_{0}e^{-2^{2j}(t-s)}\|D^{\alpha}h_{j}\|_{L^{2}}ds\\
+\sqrt{\kappa}\sup_{j}2^{(-\sigma+1) j}\int^{t}_{0}e^{-2^{2j}(t-s)}\|[D^{\alpha}\ddj,\tilde{\psi}(\kappa^{-\frac{1}{2}}a)](\Delta a+\mathrm{div}\mathcal{Q}u)\|_{L^{2}}ds.
\end{multline}
Very similarly, we obtain for $\cP u$ such that
\begin{multline}\label{decay2}
\|D^{\alpha}\cP u\|_{\dot{B}^{-\sigma}_{2,\infty}}\lesssim t^{-\frac{|\alpha|}{2}}\|\cP u_{0}\|_{\dot{B}^{-\sigma}_{2,\infty}}
+\sup_{j}2^{-\sigma j}\int^{t}_{0}e^{-2^{2j}(t-s)}\|D^{\alpha}\cP g_{j}\|_{L^{2}}ds
\end{multline}
and we conclude with
\begin{eqnarray}\label{linear decay}
\|D^{\alpha}(\frac{1}{\sqrt{\kappa}}a,\nabla a, u)\|_{\dot{B}^{-\sigma}_{2,\infty}}
\lesssim t^{-\frac{|\alpha|}{2}}\|(\frac{1}{\sqrt{\kappa}}a_{0},\nabla a_{0}, u_{0})\|_{\dot{B}^{-\sigma}_{2,\infty}}+I_{1}+I_{2}
\end{eqnarray}
where $I_{1}$ represents nonlinearities of integral of time on $[0,\frac{t}{2}]$ while $I_{2}$ is the integral on $[\frac{t}{2},t]$. Now we start to deal with different time integral intervals respectively.
\subsubsection{Nonlinear estimates on $[0,\frac{t}{2}]$}
The estimates on $[0,\frac{t}{2}]$ is not problematic, we shall take the term $\nabla(u\cdot\nabla a)$ in $h$ as an example where we have
\begin{eqnarray}
&&\sup_{j}2^{-\sigma j}\int^{\frac{t}{2}}_{0}e^{-2^{2j}(t-s)}\|D^{\alpha}\ddj\nabla(u\cdot\nabla a)\|_{L^{2}}ds\\
\nonumber&\lesssim&\sup_{j}2^{-\sigma j}\int^{\frac{t}{2}}_{0}(t-s)^{-(|\alpha|+\sigma)}((t-s)2^{2j})^{|\alpha|+\sigma} e^{-2^{2j}(t-s)}\|\ddj\nabla(u\cdot\nabla a)\|_{L^{2}}ds\\
\nonumber&\lesssim&t^{-\frac{|\alpha|}{2}}\|u\cdot\nabla a\|_{\tilde{L}^{1}_{t}(\dot{B}^{-\sigma+1}_{2,\infty})}.
\end{eqnarray}
Since that
$$\|u\cdot\nabla a\|_{\tilde{L}^{1}_{t}(\dot{B}^{-\sigma+1}_{2,\infty})}\lesssim\|u\|_{L^{\infty}_{t}(\dot{B}^{-\sigma}_{2,\infty})}\|\nabla a\|_{L^{1}_{t}(\dot{B}^{\frac{d}{2}+1}_{2,1})}
+\|\nabla a\|_{L^{\infty}_{t}(\dot{B}^{-\sigma}_{2,\infty})}\|u\|_{L^{1}_{t}(\dot{B}^{\frac{d}{2}+1}_{2,1})}.$$
Therefore, calculations as in Section 3 allows us to have
\begin{multline*}
\sup_{j}2^{-\sigma j}\int^{\frac{t}{2}}_{0}e^{-2^{2j}(t-s)}\|D^{\alpha}\ddj(u\cdot\nabla a)\|_{L^{2}}ds\\
\lesssim t^{-\frac{|\alpha|}{2}}\|(\nabla a,u)\|_{L^\infty_{t}(\dot{B}^{-\sigma}_{2,\infty})}(\mathcal{E}_{t}+\mathcal{W}_{t}+\mathcal{V}_{t}).
\end{multline*}
So does estimates for those semi-linear ones $u\cdot\nabla u, \nabla(|\nabla a|^{2})$. In terms of the pressure term $\kappa^{-\frac{1}{2}}\tilde{G}(\kappa^{-\frac{1}{2}}a)\nabla a$, we have
\begin{eqnarray*}
\sup_{j}2^{-\sigma j}\int^{\frac{t}{2}}_{0}e^{-2^{2j}(t-s)}\|D^{\alpha}\ddj(\tilde{G}(\kappa^{-\frac{1}{2}}a)\nabla a)\|_{L^{2}}ds
\lesssim t^{-\frac{|\alpha|}{2}}\|\tilde{G}(\kappa^{-\frac{1}{2}}a)\nabla a\|_{\tilde{L}^1_{t}(\dot{B}^{-\sigma}_{2,\infty})}\\
\lesssim \sqrt{\kappa} t^{-\frac{|\alpha|}{2}}\|\frac{a}{\sqrt{\kappa}}\|_{L^\infty_{t}(\dot{B}^{-\sigma}_{2,\infty})}\|\frac{a}{\sqrt{\kappa}}\|_{L^1_{t}(\dot{B}^{\frac{d}{2}+1}_{2,1})}
\end{eqnarray*}
which yields
\begin{multline*}
\kappa^{-\frac{1}{2}}\sup_{j}2^{-\sigma j}\int^{\frac{t}{2}}_{0}e^{-2^{2j}(t-s)}\|D^{\alpha}\ddj(\tilde{G}(\kappa^{-\frac{1}{2}}a)\nabla a)\|_{L^{2}}ds
\lesssim t^{-\frac{|\alpha|}{2}}\|(\nabla a,u)\|_{L^\infty_{t}(\dot{B}^{-\sigma}_{2,\infty})}\mathcal{E}_{t}.
\end{multline*}
For quasi-linear nonlinearity $g_{2}$, we pay attention to $\mathrm{div}(2\tilde{\mu}(\kappa^{-\frac{1}{2}}a) D(u))$ where
\begin{multline*}
\|\tilde{\mu}(\kappa^{-\frac{1}{2}}a) D(u)\|_{\tilde{L}^{1}_{t}(\dot{B}^{-\sigma+1}_{2,\infty})}
\lesssim \kappa^{-\frac{1}{2}}\big(\|a\|_{L^{\infty}_{t}(\dot{B}^{-\sigma+1}_{2,\infty})}\|D(u)\|_{L^{1}_{t}(\dot{B}^{\frac{d}{2}}_{2,1})}\\
+\|D(u)\|_{L^{\infty}_{t}(\dot{B}^{-\sigma-1}_{2,\infty})}\|a\|_{L^{1}_{t}(\dot{B}^{\frac{d}{2}+2}_{2,1})}\big).
\end{multline*}
which implies
\begin{multline*}
\sup_{j}2^{-\sigma j}\int^{\frac{t}{2}}_{0}e^{-2^{2j}(t-s)}\|D^{\alpha}\ddj\mathrm{div}(2\tilde{\mu}(\kappa^{-\frac{1}{2}}a) D(u))\|_{L^{2}}ds\\
\lesssim t^{-\frac{|\alpha|}{2}}\|(\frac{1}{\sqrt{\kappa}}a,\nabla a,u)\|_{L^\infty_{t}(\dot{B}^{-\sigma}_{2,\infty})}(\mathcal{E}_{t}+\mathcal{W}_{t}+\mathcal{V}_{t}).
\end{multline*}
Similarly we could estimates other terms. Finally we turn to commutators, here, we shall take $[D^{\alpha}\ddj,\tilde{\psi}(\kappa^{-\frac{1}{2}}a)]\Delta a$ as an example while the other one enjoys exact same calculations. Indeed we have
$$[D^{\alpha}\ddj,\tilde{\psi}(\kappa^{-\frac{1}{2}}a)]\Delta a=D^{\alpha}\ddj(\tilde{\psi}(\kappa^{-\frac{1}{2}}a)\Delta a)-\tilde{\psi}(\kappa^{-\frac{1}{2}}a)D^{\alpha}\ddj\Delta a.$$
The former one enjoys similar calculations as quasi-linear ones, for the latter one, we have
\begin{multline*}
\sqrt{\kappa}\sup_{j}2^{-\sigma j}\int^{\frac{t}{2}}_{0}e^{-2^{2j}(t-s)}\|\tilde{\psi}(\kappa^{-\frac{1}{2}}a)D^{\alpha}\ddj\Delta a\|_{L^{2}}ds\\
\lesssim \sqrt{\kappa}t^{-\frac{|\alpha|}{2}}\sup_{j}2^{-\sigma j}\int^{\frac{t}{2}}_{0}\|\tilde{\psi}(\kappa^{-\frac{1}{2}}a)\|_{L^{\infty}}ds\|\ddj\Delta a\|_{L^2}\\
\lesssim t^{-\frac{|\alpha|}{2}}\|a\|_{L^{r_{1}}_{t}L^\infty}\|\Delta a\|_{\tilde{L}^{r_{2}}_{t}(\dot{B}^{-\sigma}_{2,\infty})}
\end{multline*}
where $\frac{1}{r_{1}}=\frac{\sigma+\frac{d}{2}-1}{\sigma+\frac{d}{2}+1}, \frac{1}{r_{2}}=\frac{2}{\sigma+\frac{d}{2}+1}$. Keep in mind the following interpolation
$$\|a\|_{L^{r_{1}}_{t}L^\infty}\lesssim\|\nabla a\|_{L^{r_{1}}_{t}(\dot{B}^{\frac{d}{2}-1}_{2,1})}
\lesssim\|\nabla a\|^{\theta_{1}}_{L^{\infty}_{t}(\dot{B}^{-\sigma}_{2,\infty})}\|\nabla a\|^{1-\theta_{1}}_{L^{1}_{t}(\dot{B}^{\frac{d}{2}+1}_{2,1})}$$
and
$$\|\Delta a\|_{\tilde{L}^{r_{2}}_{t}(\dot{B}^{-\sigma}_{2,\infty})}\lesssim\|\nabla a\|^{\theta_{2}}_{L^{\infty}_{t}(\dot{B}^{-\sigma}_{2,\infty})}\|\nabla a\|^{1-\theta_{2}}_{L^{1}_{t}(\dot{B}^{\frac{d}{2}+1}_{2,1})}$$
with $\theta_{1}=1-\frac{1}{r_{1}}, \theta_{2}=1-\frac{1}{r_{2}}$, we arrive at
\begin{multline*}
\sqrt{\kappa}\sup_{j}2^{-\sigma j}\int^{\frac{t}{2}}_{0}e^{-2^{2j}(t-s)}\|\tilde{\psi}(\kappa^{-\frac{1}{2}}a)D^{\alpha}\ddj\Delta a\|_{L^{2}}ds\\
\lesssim t^{-\frac{|\alpha|}{2}}\|\nabla a\|_{L^{\infty}_{t}(\dot{B}^{-\sigma}_{2,\infty})}(\mathcal{E}_{t}+\mathcal{W}_{t}+\mathcal{V}_{t}).
\end{multline*}
Similar calculations on another commutator allows us to finish nonlinear estimates on $[0,\frac{t}{2}]$ and we finish with
\begin{eqnarray}\label{I1}
I_{1}&\lesssim&t^{-\frac{|\alpha|}{2}}\|(\frac{1}{\sqrt{\kappa}}a,\nabla a, u)\|_{L^{\infty}_{t}(\dot{B}^{-\sigma}_{2,\infty})}(\mathcal{E}_{t}+\mathcal{W}_{t}+\mathcal{V}_{t}).
\end{eqnarray}

\subsubsection{Nonlinear estimates on $[\frac{t}{2},t]$}
Next we term to bound $[\frac{t}{2},t]$. Again, we start with $u\cdot\nabla a=\mathcal{Q}u\cdot\nabla a+(\mathcal{P}u-v)\cdot\nabla a+v\cdot\nabla a$, for the first one, we have
\begin{eqnarray}
\nonumber&&\sup_{j}2^{-\sigma j}\int^{t}_{\frac{t}{2}}e^{-2^{2j}(t-s)}\|D^{\alpha}\ddj\nabla(\mathcal{Q}u\cdot\nabla a)\|_{L^{2}}ds\\
\nonumber&\lesssim&\sup_{j}2^{-\sigma j}\big(\int^{t}_{\frac{t}{2}}e^{-2^{2j}(t-s)}ds\big)^{\frac{1}{2}}\|D^{\alpha}\ddj\nabla(\mathcal{Q}u\cdot\nabla a)\|_{L^{2}_{[\frac{t}{2},t]}L^{2}}.
\end{eqnarray}
Observe that for fixed $c>0$ and $\gamma\in[1,\infty]$
$$\big(\int^{t}_{\frac{t}{2}}e^{-c2^{2j}(t-s)}ds\big)^{\frac{1}{\gamma}}\leq C2^{-\frac{2}{\gamma}j}$$
for some positive $C$, there naturally holds
\begin{eqnarray}
\sup_{j}2^{-\sigma j}\int^{t}_{\frac{t}{2}}e^{-2^{2j}(t-s)}\|D^{\alpha}\ddj\nabla(\mathcal{Q}u\cdot\nabla a)\|_{L^{2}}ds
\lesssim\|D^{\alpha}(\mathcal{Q}u\cdot\nabla a)\|_{\tilde{L}^{2}_{[\frac{t}{2},t]}(\dot{B}^{-\sigma}_{2,\infty})}.
\end{eqnarray}
Now Bony decomposition
$$\mathcal{Q}u\cdot\nabla a=T_{\mathcal{Q}u}\nabla a+R(\mathcal{Q}u,\nabla a)+T_{\nabla a}\mathcal{Q}u,$$
we have for all $\alpha\geq\sigma$
\begin{eqnarray}
\|D^{\alpha}T_{\mathcal{Q}u}\nabla a\|_{\tilde{L}^{2}_{[\frac{t}{2},t]}(\dot{B}^{-\sigma}_{2,\infty})}
\lesssim\|\mathcal{Q}u\|_{L^{2}_{t}L^\infty}\|D^{\alpha}\nabla a\|_{L^{\infty}_{[\frac{t}{2},t]}(\dot{B}^{-\sigma}_{2,\infty})};
\end{eqnarray}
\begin{eqnarray}
\|D^{\alpha}R(\mathcal{Q}u,\nabla a)\|_{\tilde{L}^{2}_{[\frac{t}{2},t]}(\dot{B}^{-\sigma}_{2,\infty})}
\lesssim\|\mathcal{Q}u\|_{\tilde{L}^{2}_{t}(\dot{B}^{0}_{\infty,1})}\|D^{\alpha}\nabla a\|_{L^{\infty}_{[\frac{t}{2},t]}(\dot{B}^{-\sigma}_{2,\infty})};
\end{eqnarray}
\begin{eqnarray}
\|D^{\alpha}T_{\nabla a}\mathcal{Q}u\|_{\tilde{L}^{2}_{[\frac{t}{2},t]}(\dot{B}^{-\sigma}_{2,\infty})}
\lesssim\|\nabla a\|_{L^{2}_{t}L^\infty}\|D^{\alpha}\mathcal{Q}u\|_{L^{\infty}_{[\frac{t}{2},t]}(\dot{B}^{-\sigma}_{2,\infty})}.
\end{eqnarray}
Now for the first one, if $\alpha\geq\sigma$, there holds
\begin{eqnarray}
\nonumber\|D^{\alpha}\mathcal{Q}u\cdot\nabla a\|_{\tilde{L}^{2}_{[\frac{t}{2},t]}(\dot{B}^{-\sigma}_{2,\infty})}
&\lesssim&t^{-\frac{|\alpha|}{2}}\|\nabla a\|_{\tilde{L}^{2}(\dot{B}^{0}_{\infty,1})}\sup_{[\frac{t}{2},t]}s^{\frac{\alpha}{2}}\|D^{\alpha}\mathcal{Q}u\|_{\dot{B}^{-\sigma}_{2,\infty}}.
\end{eqnarray}
Thanks to (\ref{infty}), we deduce
\begin{eqnarray}\label{popo}
\sup_{j}2^{-\sigma j}\int^{t}_{\frac{t}{2}}e^{-2^{2j}(t-s)}\|D^{\alpha}\ddj\nabla(\mathcal{Q}u\cdot\nabla a)\|_{L^{2}}ds&\lesssim&\kappa^{-\delta}t^{-\frac{|\alpha|}{2}}\mathcal{D}_{t}\mathcal{X}_{t}.
\end{eqnarray}

Now we turn to consider $(\mathcal{P}u-v)\cdot\nabla a$, in fact, it is similarly handled as (\ref{popo}) where
\begin{eqnarray*}
\sup_{j}2^{-\sigma j}\int^{t}_{\frac{t}{2}}e^{-2^{2j}(t-s)}\|D^{\alpha}\ddj\nabla((\mathcal{P}u-v)\cdot\nabla a)\|_{L^{2}}ds&\lesssim&t^{-\frac{|\alpha|}{2}}\mathcal{W}_{t}\mathcal{X}_{t}.
\end{eqnarray*}

As for $v\cdot\nabla a$, again we have Bony decomposition
$$v\cdot\nabla a=T_{v}\nabla a+R(v,\nabla a)+T_{\nabla a}v.$$
For the first one, one shall decompose the frequency into
$$\ddj T_{v}\nabla a=\dot{S}_{j}v\ddj \nabla a=\sum_{2^{j-l}\leq\kappa^{\eta}}\dot{\Delta}_{l}v\ddj \nabla a+\sum_{2^{j-l}\geq\kappa^{\eta}}\dot{\Delta}_{l}v\ddj \nabla a$$
and it holds that
\begin{eqnarray}
&&\sup_{j}2^{-\sigma j}\int^{t}_{\frac{t}{2}}e^{-2^{2j}(t-s)}\|D^{\alpha} \nabla(\sum_{2^{j-l}\leq\kappa^{\eta}}\dot{\Delta}_{l}v\ddj \nabla a)\|_{L^{2}}\\
\nonumber&\lesssim&\sup_{j}2^{-\sigma j}\sum_{2^{j-l}\leq\kappa^{\eta}}\|\dot{\Delta}_{l}v\|_{L^{\infty}_{[\frac{t}{2},t]}L^{2}}\|\ddj \nabla D^{\alpha} a\|_{L^{2}_{[\frac{t}{2},t]}L^{\infty}}\\
\nonumber&\lesssim&\sup_{j}\sum_{2^{j-l}\leq\kappa^{\eta}}2^{(\alpha-\sigma)(j-l)}2^{(\alpha-\sigma)l}\|\dot{\Delta}_{l}v\|_{L^{\infty}_{[\frac{t}{2},t]}L^{2}}2^{-\alpha j}\|\ddj \nabla D^{\alpha} a\|_{L^{2}_{[\frac{t}{2},t]}L^{\infty}}\\
\nonumber&\lesssim&\kappa^{(\alpha-\sigma)\eta}\|D^{\alpha}v\|_{L^{\infty}_{[\frac{t}{2},t]}(\dot{B}^{-\sigma}_{2,\infty})}\| \nabla a\|_{L^{2}_{[\frac{t}{2},t]}L^\infty}.
\end{eqnarray}
We conclude for any $t\geq1$
\begin{multline}
\sum_{j}2^{-\sigma j}\int^{t}_{\frac{t}{2}}e^{-2^{2j}(t-s)}\|\nabla(\sum_{j-l\leq\kappa^{\eta}}\dot{\Delta}_{l}v\ddj \nabla D^{\alpha} a)\|_{L^{2}}ds
\lesssim\kappa^{(\alpha-\sigma)\eta-\delta}t^{-\frac{|\alpha|}{2}}\mathcal{Y}_{t}\mathcal{D}_{t}.
\end{multline}
On the other hand, we have
\begin{eqnarray}
&&\sup_{j}2^{-\sigma j}\int^{t}_{\frac{t}{2}}e^{-2^{2j}(t-s)}\|\nabla(\sum_{j-l\geq\kappa^{\eta}}\dot{\Delta}_{l}v\ddj \nabla D^{\alpha} a)\|_{L^{2}}ds\\
\nonumber&\lesssim&\sup_{j}2^{(-\sigma-1) j}\sum_{j-l\geq\kappa^{\eta}}\|\dot{\Delta}_{l}v\|_{L^{\infty}_{[\frac{t}{2},t]}L^{\infty}}\|\ddj \nabla D^{\alpha} a\|_{L^{\infty}_{[\frac{t}{2},t]}L^{2}}\\
\nonumber&\lesssim&\sup_{j}\sum_{2^{j-l}\geq\kappa^{\eta}}2^{-(j-l)}2^{-l}\|\dot{\Delta}_{l}v\|_{L^{\infty}_{[\frac{t}{2},t]}L^{\infty}}2^{-\sigma j}\|\ddj \nabla D^{\alpha} a\|_{L^{\infty}_{[\frac{t}{2},t]}L^{2}}\\
\nonumber&\lesssim&\kappa^{-\eta}\|v\|_{L^{\infty}_{[\frac{t}{2},t]}(\dot{B}^{\frac{d}{2}-1}_{2,\infty})}\| \nabla D^{\alpha} a\|_{L^{\infty}_{[\frac{t}{2},t]}(\dot{B}^{-\sigma}_{2,\infty})}
\lesssim\kappa^{-\eta}t^{-\frac{|\alpha|}{2}}t^{-\epsilon}\mathcal{V}_{t}\mathcal{X}_{t}
\end{eqnarray}
where $\epsilon=\frac{d}{2}-1-\sigma>0$. Finally, we have
\begin{eqnarray}
&&\sup_{j}2^{-\sigma j}\int^{t}_{\frac{t}{2}}e^{-2^{2j}(t-s)}\|\ddj D^{\alpha} \nabla(R(v,\nabla a)+T_{\nabla a}v)\|_{L^{2}}ds\\
\nonumber&\lesssim&\|D^{\alpha} R(v,\nabla a)+D^{\alpha} T_{\nabla a}v\|_{L^{2}_{[\frac{t}{2},t]}(\dot{B}^{-\sigma}_{2,\infty})}\\
\nonumber&\lesssim&\|\nabla a\|_{L^{2}_{[\frac{t}{2},t]}(\dot{B}^{0}_{\infty,1})}\|D^{\alpha} v\|_{L^{\infty}_{[\frac{t}{2},t]}(\dot{B}^{-\sigma}_{2,\infty})}\\
\nonumber&\lesssim&t^{-\frac{|\alpha|}{2}}\kappa^{-\delta}\mathcal{D}_{t}\mathcal{X}_{t}.
\end{eqnarray}

Consequently, by selecting $\eta=\frac{\delta}{\alpha-\sigma}$, we finally have
\begin{multline}
\sup_{j}2^{-\sigma j}\int^{t}_{0}e^{-2^{2j}(t-s)}\|D^{\alpha}\ddj\nabla(v\cdot\nabla a)\|_{L^{2}}ds\\ \lesssim
t^{-\frac{|\alpha|}{2}}\big(\mathcal{Y}_{t}\mathcal{D}_{t}+\kappa^{-\frac{\delta}{\alpha-\sigma}}t^{-\epsilon}\mathcal{V}_{t}\mathcal{X}_{t}+\kappa^{-\delta}\mathcal{D}_{t}\mathcal{X}_{t}\big).
\end{multline}

So does estimates for those semi-linear ones $\nabla(|\nabla a|^{2}),\kappa^{-\frac{1}{2}}\tilde{G}(\kappa^{-\frac{1}{2}}a)\nabla a$ while for $v\cdot\nabla v$, there holds
 \begin{multline}
\sup_{j}2^{-\sigma j}\int^{t}_{0}e^{-2^{2j}(t-s)}\|D^{\alpha}\ddj(v\cdot\nabla v)\|_{L^{2}}ds\lesssim\|D^{\alpha} v^{2}\|_{L^{2}_{[\frac{t}{2},t]}(\dot{B}^{-\sigma}_{2,\infty})}\\
\lesssim\|v\|_{\tilde{L}^{2}_{[\frac{t}{2},t]}\dot{B}^{\frac{d}{2}}_{2,1}}\|D^{\alpha} v\|_{L^{\infty}_{[\frac{t}{2},t]}(\dot{B}^{-\sigma}_{2,\infty})}
\lesssim t^{-\frac{|\alpha|}{2}}\mathcal{V}_{t}\mathcal{Y}_{t}.
\end{multline}
Next we focus on those quasi-linear ones and take $\nabla(\tilde{\lambda}(\kappa^{-\frac{1}{2}}a)\mathrm{div}u)$ as an example. We have
\begin{eqnarray}
&&\sup_{j}2^{-\sigma j}\int^{t}_{\frac{t}{2}}e^{-2^{2j}(t-s)}\|\ddj D^{\alpha}\nabla(\tilde{\lambda}(\kappa^{-\frac{1}{2}}a)\mathrm{div}u)\|_{L^{2}}ds\\
\nonumber&\lesssim&\|D^{\alpha}(\tilde{\lambda}(\kappa^{-\frac{1}{2}}a)\mathrm{div}u)\|_{L^{\infty}_{[\frac{t}{2},t]}(\dot{B}^{-\sigma-1}_{2,\infty})}.
\end{eqnarray}
Since that
\begin{eqnarray}
&&\|D^{\alpha}(\tilde{\lambda}(\kappa^{-\frac{1}{2}}a)\mathrm{div}u)\|_{L^{\infty}_{[\frac{t}{2},t]}(\dot{B}^{-\sigma-1}_{2,\infty})}\\
\nonumber&\lesssim&\|D^{\alpha-1}\tilde{\lambda}(\kappa^{-\frac{1}{2}}a)\mathrm{div}\mathcal{Q}u\|_{L^{\infty}_{[\frac{t}{2},t]}(\dot{B}^{-\sigma}_{2,\infty})}
+\|\tilde{\lambda}(\kappa^{-\frac{1}{2}}a)D^{\alpha-1}\mathrm{div}\mathcal{Q}u\|_{L^{\infty}_{[\frac{t}{2},t]}(\dot{B}^{-\sigma}_{2,\infty})}\\
\nonumber&\lesssim&\kappa^{-\frac{1}{2}}\big(\|a\|_{L^{\infty}_{[\frac{t}{2},t]}(\dot{B}^{\frac{d}{2}}_{2,1})}\|D^{\alpha}\mathcal{Q}u\|_{L^{\infty}_{[\frac{t}{2},t]}(\dot{B}^{-\sigma}_{2,\infty})}
+\|D^{\alpha-1}a\|_{L^{\infty}_{[\frac{t}{2},t]}(\dot{B}^{\frac{d}{2}}_{2,1})}\|\mathrm{div}\mathcal{Q}u\|_{L^{\infty}_{[\frac{t}{2},t]}(\dot{B}^{-\sigma}_{2,\infty})}\big).
\end{eqnarray}
we have
\begin{eqnarray}
\sup_{j}2^{-\sigma j}\int^{t}_{\frac{t}{2}}e^{-2^{2j}(t-s)}\|\ddj D^{\alpha}\nabla(\tilde{\lambda}(\kappa^{-\frac{1}{2}}a)\mathrm{div}u)\|_{L^{2}}ds
\lesssim\kappa^{-\frac{1}{2}}t^{-\frac{|\alpha|}{2}}\mathcal{D}_{t}\mathcal{X}_{t}.
\end{eqnarray}

At this end, we turn to commutator. It could be written as
\begin{multline*}
\!\!\!\!\!\![D^{\alpha}\ddj,\tilde{\psi}(\kappa^{-\frac{1}{2}}a)]\Delta a=[D^{\alpha}\ddj,T_{\tilde{\psi}(\kappa^{-\frac{1}{2}}a)}]\Delta a
+D^{\alpha}\ddj(T_{\Delta a}\tilde{\psi}(\kappa^{-\frac{1}{2}}a))+T_{D^{\alpha}\ddj\Delta a}\tilde{\psi}(\kappa^{-\frac{1}{2}}a)\\
+D^{\alpha}\ddj R(\Delta a,\tilde{\psi}(\kappa^{-\frac{1}{2}}a))+R(D^{\alpha}\ddj \Delta a,\tilde{\psi}(\kappa^{-\frac{1}{2}}a))
\triangleq\sum^{5}_{i=1}R_{i}.
\end{multline*}

For $R_{1}$, we have
$$[D^{\alpha}\ddj,T_{\tilde{\psi}(\kappa^{-\frac{1}{2}}a)}]\Delta a=\sum_{k\leq j-2}[D^{\alpha}\ddj,\dot{\Delta}_{k}\tilde{\psi}(\kappa^{-\frac{1}{2}}a)]\Delta \ddj a.$$
Classical commutator estimates under localization indicates
\begin{eqnarray}
\|\sum_{k\leq j-2}[D^{\alpha}\ddj,\dot{\Delta}_{k}\tilde{\psi}(\kappa^{-\frac{1}{2}}a)]\Delta \ddj a\|_{L^{2}}
\nonumber&\lesssim&2^{(\alpha-1) j}\|\nabla\tilde{\psi}(\kappa^{-\frac{1}{2}}a)\|_{L^\infty}\|\Delta \ddj a\|_{L^{2}}
\end{eqnarray}
which implies
\begin{multline*}
\sqrt{\kappa}\sup_{j}2^{(-\sigma+1) j}\int^{\frac{t}{2}}_{0}e^{-2^{2j}(t-s)}\|R_{1}\|_{L^{2}}ds\\
\lesssim\sup_{j}2^{(-\sigma+\alpha-1)j}\|\Delta \ddj a\|_{L^\infty_{[\frac{t}{2},t]}L^{2}}\|\nabla a\|_{L^2_{t}L^\infty}
\lesssim \kappa^{-\delta}t^{-\frac{|\alpha|}{2}}\mathcal{D}_{t}\mathcal{X}_{t}.
\end{multline*}
As for $R_{2}$ and $R_{4}$, we have
\begin{multline*}
\sqrt{\kappa}\sup_{j}2^{(-\sigma+1)j}\int^{\frac{t}{2}}_{0}e^{-2^{2j}(t-s)}\|R_{2}+R_{4}\|_{L^{2}}ds
\lesssim\|R_{2}+R_{4}\|_{\tilde{L}^{2}_{[\frac{t}{2},t]}(\dot{B}^{-\sigma}_{2,\infty})}\\
\lesssim\|D^{\alpha}\Delta a\|_{L^\infty_{[\frac{t}{2},t]}(\dot{B}^{-\sigma-1}_{2,\infty})}\| a\|_{\tilde{L}^2_{[\frac{t}{2},t]}(\dot{B}^{1}_{\infty,1})}
\lesssim \kappa^{-\delta}t^{-\frac{|\alpha|}{2}}\mathcal{D}_{t}\mathcal{X}_{t}.
\end{multline*}
provided $\alpha\geq\sigma$. In consider of $R_{3}, R_{5}$, we actually have
\begin{multline}
\|R_{3}+ R_{5}\|_{L^{2}}
\lesssim\|D^{\alpha}\ddj\Delta a\|_{\dot{B}^{-1}_{2,\infty}}\|\tilde{\psi}(\kappa^{-\frac{1}{2}}a)\|_{\dot{B}^{\frac{d}{p}+1}_{p,1}}
\lesssim\kappa^{-\frac{1}{2}}\|D^{\alpha}\ddj\nabla a\|_{L^2}\|\nabla a\|_{\dot{B}^{\frac{d}{p}}_{p,1}}.
\end{multline}
Therefore, it holds
\begin{multline*}
\sqrt{\kappa}\sup_{j}2^{(-\sigma+1)j}\int^{\frac{t}{2}}_{0}e^{-2^{2j}(t-s)}\|R_{3}+R_{5}\|_{L^{2}}ds\\
\lesssim\|D^{\alpha}\nabla a\|_{L^\infty_{[\frac{t}{2},t]}(\dot{B}^{-\sigma}_{2,\infty})}\| a\|_{\tilde{L}^2_{[\frac{t}{2},t]}(\dot{B}^{\frac{d}{p}}_{p,1})}
\lesssim \kappa^{-\delta}t^{-\frac{|\alpha|}{2}}\mathcal{D}_{t}\mathcal{X}_{t}.
\end{multline*}
Above inequalities indicate that
\begin{eqnarray}\label{I2}
I_{2}&\lesssim&t^{-\frac{|\alpha|}{2}}\big((\mathcal{V}_{t}+\mathcal{D}_{t})\mathcal{Y}_{t}+\mathcal{W}_{t}\mathcal{X}_{t}
+\kappa^{-\frac{\delta}{\alpha-\sigma}}t^{-\epsilon}\mathcal{V}_{t}\mathcal{X}_{t}+\kappa^{-\delta}\mathcal{D}_{t}\mathcal{X}_{t}\big).
\end{eqnarray}
Finally, combining (\ref{linear decay}) with (\ref{I1}), (\ref{I1}) allows us to arrive at  (\ref{3}) by taking the supremum norm of $t$.
\end{proof}

Since that $\mathcal{W}_{t}\leq2C\kappa^{-\delta} e^{C\mathcal{V}_{t}}\mathcal{V}^{2}_{T}$, the fact $\kappa^{-1}\ll1$ enable us to find
$$\mathcal{W}_{t}+\kappa^{-\frac{\delta}{\alpha-\sigma}}t^{-\epsilon}\mathcal{V}_{t}+\kappa^{-\delta}\mathcal{D}_{t}\leq\frac{1}{2}.$$
Therefore notice the fact
$$\mathcal{Y}_{t}\lesssim \sup_{t\in[0,T]}t^{\frac{|\alpha|}{2}}\|\Lambda^{\frac{\alpha-\sigma}{2}}v\|_{L^2}\leq C,$$
where we utilize (\ref{incompressible decay}), the boundedness of $\|(\frac{1}{\sqrt{\kappa}}a,\nabla a, u)\|_{L^{\infty}_{t}(\dot{B}^{-\sigma}_{2,\infty})}$ and $\mathcal{V}_{t},\mathcal{Y}_{t},\mathcal{E}_{t},\mathcal{D}_{t}$ indicates $\mathcal{X}_{t}\leq C_{0}$ where $C_{0}$ depends on initial data and conclude with Theorem \ref{thm2.2}.

\section{Appendix}
In this appendix, we would recall some classical theory concerns Fourier localization technique. The Fourier transform of a function $f\in\mathcal{S}$ (the Schwarz class) is denoted by
$$\widehat{f}(\xi)=\mathcal{F}[f](\xi):=\int_{\mathbb{R}^{d}}f(x)e^{-i\xi\cdot x}dx.$$
For $ 1\leq p\leq \infty$, we denote by $L^{p}=L^{p}(\mathbb{R}^{d})$ the usual Lebesgue space on $\mathbb{R}^{d}$ with the norm $\|\cdot\|_{L^{p}}$.

For convenience of reader, we would like to recall the Littlewood-Paley decomposition, Besov spaces and related analysis tools. The reader is referred to Chap. 2 and Chap. 3 of \cite{BCD} for more details. Let $\chi$ be a smooth function valued in $[0,1]$, such that $\chi$ is supported in the ball
$\mathbf{B}(0,\frac{4}{3})=\{\xi\in\mathbb{R}^{d}:|\xi|\leq\frac{4}{3}\}$. Set $\varphi(\xi)=\chi(\xi/2)-\chi(\xi)$. Then $\varphi$
is supported in the shell $\mathbf{C}(0,\frac{3}{4},\frac{8}{3})=\{\xi\in\mathbb{R}^{d}:\frac{3}{4}\leq|\xi|\leq\frac{8}{3}\}$ so that
$$\sum_{q\in\mathbb{Z}}\varphi(2^{-q}\xi)=1, \quad \forall\xi\in\mathbb{R}^{d}\backslash\{{0}\}.$$

For any tempered distribution $f\in\mathcal{S}'$, one can define the homogeneous dyadic blocks and homogeneous low-frequency cut-
off operators:
\begin{eqnarray*}
&&\dot{\Delta}_{q}f:=\varphi(2^{-q}D)f=\mathcal{F}^{-1}(\varphi(2^{-q}\xi)\mathcal{F}f), \quad q\in\mathbb{Z};
\end{eqnarray*}
\begin{eqnarray*}
&&\dot{S}_{q}f:=\chi(2^{-q}D)f=\mathcal{F}^{-1}(\chi(2^{-q}\xi)\mathcal{F}f), \quad q\in\mathbb{Z}.
\end{eqnarray*}
Furthermore, we have the formal homogeneous decomposition as follows
$$f=\sum_{q\in\mathbb{Z}}\dot{\Delta}_{q}f.$$
Also, throughout the paper, $f^{h}$ and $f^{\ell}$ represent the high frequency part and low frequency part of $f$ respectively where
$$f^{\ell}\triangleq\dot{S}_{q_{0}}f;\quad f^{h}\triangleq(1-\dot{S}_{q_{0}})f$$
with some given constant $q_{0}$.

Denote by $\mathcal{S}'_{0}:=\mathcal{S'}/\mathcal{P}$ the tempered distributions modulo polynomials $\mathcal{P}$. As we known, the homogeneous
Besov spaces can be characterised in terms of the above spectral cut-off blocks.

\subsection{{\bf Homogeneous Besov space}}
\hspace*{\fill}
\begin{defn}\label{defn2.1}
For $s\in \mathbb{R}$ and $1\leq p,r\leq \infty$, the homogeneous Besov spaces $\dot{B}^s_{p,r}$ are defined by
$$\dot{B}^s_{p,r}:=\Big\{f\in \mathcal{S}'_{0}:\|f\|_{\dot{B}^s_{p,r}}<\infty  \Big\} ,$$
where
\begin{equation*}
\|f\|_{\dot{B}^s_{p,r}}:=\Big(\sum_{q\in\mathbb{Z}}(2^{qs}\|\dot{\Delta}_qf\|_{L^{p}})^{r}\Big)^{1/r}
\end{equation*}
with the usual convention if $r=\infty$.
\end{defn}

We often use the following classical properties of Besov spaces (see \cite{BCD}):

$\bullet$ \ \emph{Scaling invariance:} For any $\sigma\in \mathbb{R}$ and $(p,r)\in
[1,\infty ]^{2}$, there exists a constant $C=C(\sigma,p,r,d)$ such that for all $\lambda >0$ and $f\in \dot{B}_{p,r}^{\sigma}$, we have
$$
C^{-1}\lambda ^{\sigma-\frac {d}{p}}\|f\|_{\dot{B}_{p,r}^{\sigma}}
\leq \|f(\lambda \,\cdot)\|_{\dot{B}_{p,r}^{\sigma}}\leq C\lambda ^{\sigma-\frac {d}{p}}\|f\|_{\dot{B}_{p,r}^{\sigma}}.
$$

$\bullet$ \ \emph{Completeness:} $\dot{B}^{\sigma}_{p,r}$ is a Banach space whenever $
\sigma<\frac{d}{p}$ or $\sigma\leq \frac{d}{p}$ and $r=1$.

$\bullet$ \ \emph{Interpolation:} The following inequality is satisfied for $1\leq p,r_{1},r_{2}, r\leq \infty, \sigma_{1}\neq \sigma_{2}$ and $\theta \in (0,1)$:
$$\|f\|_{\dot{B}_{p,r}^{\theta \sigma_{1}+(1-\theta )\sigma_{2}}}\lesssim \|f\| _{\dot{B}_{p,r_{1}}^{\sigma_{1}}}^{\theta} \|f\|_{\dot{B}_{p,r_2}^{\sigma_{2}}}^{1-\theta }$$
with $\frac{1}{r}=\frac{\theta}{r_{1}}+\frac{1-\theta}{r_{2}}$.

$\bullet$ \ \emph{Action of Fourier multipliers:} If $F$ is a smooth homogeneous of
degree $m$ function on $\mathbb{R}^{d}\backslash \{0\}$ then
$$F(D):\dot{B}_{p,r}^{\sigma}\rightarrow \dot{B}_{p,r}^{\sigma-m}.$$


The embedding properties will be used several times throughout the paper.
\begin{prop} \label{Prop2.1}
\begin{itemize}
\item For any $p\in[1,\infty]$ we have the continuous embedding
$\dot {B}^{0}_{p,1}\hookrightarrow L^{p}\hookrightarrow \dot{B}^{0}_{p,\infty}$.
\item If $\sigma\in \mathbb{R}$, $1\leq p_{1}\leq p_{2}\leq\infty$ and $1\leq r_{1}\leq r_{2}\leq\infty,$
then $\dot {B}^{\sigma}_{p_1,r_1}\hookrightarrow
\dot {B}^{\sigma-d\,(\frac{1}{p_{1}}-\frac{1}{p_{2}})}_{p_{2},r_{2}}$.
\item The space $\dot {B}^{\frac {d}{p}}_{p,1}$ is continuously embedded in the set of
bounded  continuous functions (going to zero at infinity if, additionally, $p<\infty$).
\end{itemize}
\end{prop}

In addition, we also recall the classical \emph{Bernstein inequality}:
\begin{equation}\label{Eq:2.6}
\|D^{k}f\|_{L^{b}}
\leq C^{1+k} \lambda^{k+d(\frac{1}{a}-\frac{1}{b})}\|f\|_{L^{a}}
\end{equation}
that holds for all function $f$ such that $\mathrm{Supp}\,\mathcal{F}f\subset\left\{\xi\in \mathbb{R}^{d}: |\xi|\leq R\lambda\right\}$ for some $R>0$
and $\lambda>0$, if $k\in\mathbb{N}$ and $1\leq a\leq b\leq\infty$.

More generally, if we assume $f$ to satisfy $\mathrm{Supp}\,\mathcal{F}f\subset\{\xi\in \mathbb{R}^{d}:
R_{1}\lambda\leq|\xi|\leq R_{2}\lambda\}$ for some $0<R_{1}<R_{2}$ and $\lambda>0$, then for any smooth
homogeneous of degree $m$ function $A$ on $\mathbb{R}^d\setminus\{0\}$ and $1\leq a\leq\infty$, we have
(see e.g. Lemma 2.2 in \cite{BCD}):
\begin{equation}\label{Eq:2.7}
\|A(D)f\|_{L^{a}}\approx\lambda^{m}\|f\|_{L^{a}}.
\end{equation}
An obvious  consequence of (\ref{Eq:2.6}) and (\ref{Eq:2.7}) is that
$\|D^{k}f\|_{\dot{B}^{s}_{p, r}}\thickapprox \|f\|_{\dot{B}^{s+k}_{p, r}}$ for all $k\in\mathbb{N}$.

Moreover, a class of mixed space-time Besov spaces are also used when studying the evolution PDEs, which were firstly
proposed by J.-Y. Chemin and N. Lerner in \cite{CL}.
\begin{defn}\label{defn2.2}
For $T>0,s\in \mathbb{R}, 1\leq r,\theta \leq \infty$, the homogeneous Chemin-Lerner spaces $\tilde{L}^\theta_{T}(\dot{B}^s_{p,r})$ are defined by
$$\tilde{L}^\theta_{T}(\dot{B}^s_{p,r}):=\Big\{f\in L^\theta(0,T;\mathcal{S}'_{0}) :\|f\|_{\tilde{L}^\theta_{T}(\dot{B}^s_{p,r})}<\infty  \Big\}, $$
where
$$\|f\|_{\tilde{L}^\theta_{T}(\dot{B}^s_{p,r})}:=\Big(\sum_{q\in \mathbb{Z}}(2^{qs}\|\dot{\Delta}_qf\|_{L^\theta_{T}(L^{p})})^{r}\Big)^{1/r} $$
with the usual convention if  $r=\infty$.
\end{defn}

The Chemin-Lerner space $\widetilde{L}^{\theta}_{T}(\dot{B}^{s}_{p,r})$ may be linked with the standard spaces $L_{T}^{\theta}(\dot{B}^{s}_{p,r})$ by means of Minkowski's inequality.
\begin{rem}\label{Rem2.1}
It holds that
$$\left\|f\right\|_{\widetilde{L}^{\theta}_{T}(\dot{B}^{s}_{p,r})}\leq\left\|f\right\|_{L^{\theta}_{T}(\dot{B}^{s}_{p,r})}\,\,\,
\mbox{if} \,\, \, r\geq\theta;\ \ \ \
\left\|f\right\|_{\widetilde{L}^{\theta}_{T}(\dot{B}^{s}_{p,r})}\geq\left\|f\right\|_{L^{\theta}_{T}(\dot{B}^{s}_{p,r})}\,\,\,
\mbox{if}\,\,\, r\leq\theta.
$$
\end{rem}

\subsection{{\bf Product estimates and composition estimates}}
\hspace*{\fill}

The product estimates in Besov spaces play a fundamental role in bounding bilinear terms in \eqref{linearized} (see \cite{BCD}).

\begin{prop}\label{prop3.2}
Let $s>0$ and $1\leq p,\,r\leq\infty$. Then $\dot{B}^{s}_{p,r}\cap L^{\infty}$ is an algebra and
$$
\|fg\|_{\dot{B}^{s}_{p,r}}\lesssim \|f\|_{L^{\infty}}\|g\|_{\dot{B}^{s}_{p,r}}+\|g\|_{L^{\infty}}\|f\|_{\dot{B}^{s}_{p,r}}.
$$
If $s_{1},s_{2}\leq\frac{d}{p}$, $s_{1}+s_{2}>d\max\{0,\frac{2}{p}-1\}$, then
$$\|ab\|_{\dot{B}^{s_{1}+s_{2}-\frac{d}{p}}_{p,1}}\lesssim\|a\|_{\dot{B}^{s_{1}}_{p,1}}\|b\|_{\dot{B}^{s_{2}}_{p,1}}.$$
If $s_{1}\leq\frac{d}{p}$, $s_{2}<\frac{d}{p}$, $s_{1}+s_{2}\geq d\max\{0,\frac{2}{p}-1\}$, then
$$\|ab\|_{\dot{B}^{s_{1}+s_{2}-\frac{d}{p}}_{p,\infty}}\lesssim\|a\|_{\dot{B}^{s_{1}}_{p,1}}\|b\|_{\dot{B}^{s_{2}}_{p,\infty}}.$$
\end{prop}

System  \eqref{linearized}  also involves compositions of functions that
are handled according to the following estimates.

\begin{prop}\label{prop2.25}
Let $F:\mathbb{R}\rightarrow \mathbb{R}$ be smooth with $F(0)=0$. For all $1\leq p,\,r\leq\infty$ and $s>0$ we have
$F(f)\in \dot {B}^{s}_{p,r}\cap L^{\infty}$ for $f\in \dot{B}^{s}_{p,r}\cap L^{\infty}$, and
$$\|F(f)\|_{\dot B^{s}_{p,r}}\leq C\|f\|_{\dot B^{s}_{p,r}}$$
with $C>0$ depending only on $\|f\|_{L^{\infty}}$, $F'$ (and higher derivatives), $s$, $p$ and $d$.

In the case $s>-d\min(\frac {1}{p},\frac {1}{p'})$ then $f\in\dot{B}^{s}_{p,r}\cap\dot {B}^{\frac {d}{p}}_{p,1}$
implies that $F(f)\in \dot{B}^{s}_{p,r}\cap\dot {B}^{\frac {d}{p}}_{p,1}$, and
$$\|F(f)\|_{\dot B^{s}_{p,r}}\leq C\|f\|_{\dot {B}^{s}_{p,r}},$$
where $C>0$ is some constant depends on $\|f\|_{\dot B^{\frac{d}{p}}_{p,1}}$, $F, s, p$ and $d$.
\end{prop}

\subsection{{\bf Proof of Proposition \ref{dispersive estimates}}}
\hspace*{\fill}

Taking advantages of interpolation, it is enough to prove the case $p=\infty$. Inspired by Young inequality and heat kernel estimates under Fourier localization, we immediately have
\begin{eqnarray}
\|e^{i\sqrt{\kappa} Ht}e^{\Delta t}f_{j}\|_{L^\infty}\lesssim  e^{-2^{2j}t}\|e^{i\sqrt{\kappa} Ht}\psi_{j}\|_{L^\infty}\|f_{j}\|_{L^1}.
\end{eqnarray}
where $\hat{\psi}(\xi)=\varphi(\xi)$. Then (\ref{infty}) is given if the following estimate holds true:
\begin{eqnarray}\label{pointwise}
\|e^{i\sqrt{\kappa} Ht}\psi_{j}\|_{L^\infty}\lesssim  (\sqrt{\kappa} t)^{-\frac{d}{2}}.
\end{eqnarray}

The (\ref{pointwise}) is proved by stationary phase method given in \cite{GPW} and we shall focus on case $d\geq2$ while $d=1$ is more direct by van der Corput's lemma, see \cite{VV}. Actually, denote by $J_{m} (r)$ the Bessel function, one could write
\begin{eqnarray}\label{Bessel}
e^{i\sqrt{\kappa} Ht}\psi_{j}=2^{jd}\int^{\infty}_{0}e^{it\tilde{H}(2^{j}r)}\varphi(r)r^{d-1}(r2^{j}|x|)^{-\frac{n-2}{2}}J_{\frac{n-2}{2}}(r2^{j}|x|)dr
\end{eqnarray}
where $\tilde{H}(r)=r\sqrt{1+\kappa r^2}$. Then it is not difficult to see
$$\tilde{H}'(r)\thicksim 1;\,\,\,\tilde{H}''(r)\thicksim \kappa r,\,\,\,\tilde{H}^{(m)}(r)\leq (\kappa r)^{1-m},\,\,\,r\leq\kappa^{-\frac{1}{2}};$$
$$\tilde{H}'(r)\thicksim \sqrt{\kappa}r;\,\,\,\tilde{H}''(r)\thicksim \sqrt{\kappa},\,\,\,\tilde{H}^{(m)}(r)\leq \sqrt{\kappa}r^{2-m},\,\,\,r\geq\kappa^{-\frac{1}{2}}.$$
Therefore, we would consider case (i). $2^{j}\leq\kappa^{-\frac{1}{2}}$ and (ii). $2^{j}\leq\kappa^{-\frac{1}{2}}$ respectively. For case (i), the corresponding behaviors are closely linked with wave operator and we start with $|x|\leq2$. In fact, denote $D_{r}=\frac{1}{it\tilde{H}'(2^j r)2^j}\frac{d}{dr}$, integral by parts in terms of $D_{r}$  immediately yields
\begin{eqnarray}\label{mmm}
e^{i\sqrt{\kappa} Ht}\psi_{j}&=&2^{jd}\int^{\infty}_{0}D^{k}_{r}\big(e^{it\tilde{H}(2^{j}r)}\big)\varphi(r)r^{d-1}(r2^{j}|x|)^{-\frac{n-2}{2}}J_{\frac{n-2}{2}}(r2^{j}|x|)dr\\
\nonumber&=&\frac{2^{jd}}{(it2^j)^k}\sum^{k}_{m=0}\sum^{m}_{l_{m}}\int^{\infty}_{0}e^{it\tilde{H}(2^{j}r)}\prod_{l_{m}}\partial^{l_{m}}_{r}\big(\frac{1}{\tilde{H}'(2^j r)}\big)\\
\nonumber&\cdot&\partial^{k-m}_{r}\big(\varphi(r)r^{d-1}(r2^{j}|x|)^{-\frac{n-2}{2}}J_{\frac{n-2}{2}}(r2^{j}|x|)\big)dr.
\end{eqnarray}
where $m=l_{1}+l_{2}+...+l_{m}$. Keep in mind that for any $m\geq0$
$$\frac{d^{m}}{d^{m}_{r}}\big(\frac{1}{\tilde{H}'(2^j r)}\big)\leq c,\,\,\,\mathrm{for}\,\,\,2^{j}\leq\kappa^{-\frac{1}{2}},$$
the vanishing property of the Bessel function at the origin indicates
\begin{eqnarray}\label{mmmmm}
|e^{i\sqrt{\kappa} Ht}\psi_{j}|&\leq&Ct^{-k}2^{j(d-k)}.
\end{eqnarray}
Hence, taking $k=\frac{d}{2}$ yields (\ref{pointwise}). For case $|x|\geq2$, we rewrite (\ref{Bessel}) into
 \begin{eqnarray}
e^{i\sqrt{\kappa} Ht}\psi_{j}=2^{jd}\int^{\infty}_{0}e^{it(\tilde{H}(2^{j}r)-r|x|)}\varphi(r)r^{d-1}h(r|x|)dr
\end{eqnarray}
where
$$\mathcal{R}(e^{ir}h(r))=cr^{-\frac{n-2}{2}}J_{\frac{n-2}{2}}(r).$$
At this moment, we start with $|x|$ fulfills $\frac{1}{2}\inf\limits_{r}t2^{j}\tilde{H}'(2^{j}r)\leq|x|\leq2\sup\limits_{r} t2^{j}\tilde{H}'(2^{j}r)$. Denote $\bar{H}(r)=\tilde{H}(2^{j}r)-r|x|$, it is clear that $\bar{H}''(r)=2^{2j}\tilde{H}''(2^{j}r)\geq\kappa 2^{3j}$, therefore, by van der Corput¡¯s lemma and pointwise estimates for $h$ (see in \cite{GPW}), there holds
\begin{eqnarray*}
|e^{i\sqrt{\kappa} Ht}\psi_{j}|&\leq&C(|t|\kappa 2^{3j})^{-\frac{1}{2}}\int^{\infty}_{0}\big|\frac{d}{dr}(\varphi(r)r^{d-1}h(r|x|))\big|dr\leq t^{-\frac{d}{2}}\kappa^{-\frac{1}{2}}2^{\frac{d-2}{2}j},
\end{eqnarray*}
which implies (\ref{pointwise}) by the fact $2^{j}\leq\kappa^{-\frac{1}{2}}$.
For $|x|\leq\frac{1}{2}\inf\limits_{r}t2^{j}\tilde{H}'(2^{j}r)$ and $2\sup\limits_{r} t2^{j}\tilde{H}'(2^{j}r)\leq|x|$, there holds
$$\big|\bar{H}'(r)\big|\geq c2^{j},\,\,\,c>0,$$
we could repeat same calculations as (\ref{mmm})-(\ref{mmmmm}) and conclude with (\ref{pointwise}).

The case $2^{j}\geq\kappa^{-\frac{1}{2}}$ is treated as low frequencies above, in which part could be regarded as a \Scho\, semi-group with parameter $\sqrt{\kappa}$. We omit the detailed proof and conclude with (\ref{pointwise}).

\noindent {\bf Acknowledgments:}\ \

The author sincerely thanks Professor Nakanishi Kenji and Professor Jiang Xu for helpful suggestions and discussions in the course of this research.

\noindent {\bf Conflicts of interest statement:}\ \

The author does not have any possible conflict of interest.

\end{document}